\documentclass[12pt,a4paper]{article}

\usepackage[utf8]{inputenc}
\usepackage{color}
\usepackage{amssymb,amsmath,amsthm,amsxtra}
\usepackage[english]{babel}
\usepackage{hyperref}
\usepackage{microtype}
\usepackage{graphicx}
\usepackage{titlepic}

\newtheorem{thm}{Theorem}
\newtheorem{lemma}[thm]{Lemma}
\newtheorem{prop}[thm]{Proposition}
\newtheorem{definition}[thm]{Definition}
\newtheorem{remark}[thm]{Remark}

\newtheorem{corollary}[thm]{Corollary}

\newcommand{\Prob}[1]{\mathbb{P}\left( #1 \right)} 
\newcommand{\ProbI}[2]{\mathbb{P}^{#1}\!\left( #2 \right)} 
\newcommand{\Probn}[1]{\mathbb{P}^{n}\!\left( #1 \right)} 
\newcommand{\Probx}[1]{\mathbb{P}^{x}\!\left( #1 \right)} 
\newcommand{\E}[1]{\mathbb{E}\left(#1\right)} 
\newcommand{\En}[1]{\mathbb{E}^{n}\left(#1\right)} 
\newcommand{\Ex}[1]{\mathbb{E}^{x}\left(#1\right)} 
\newcommand{\EI}[2]{\mathbb{E}^{#1}\left(#2\right)} 

\newcommand{\Filt}{\mathcal{F}} 

\newcommand{\ssum}[2]{\sum\limits_{#1}^{#2}}
\newcommand{\sint}[2]{\int\limits_{#1}^{#2}}
\newcommand{\sprod}[2]{\prod\limits_{#1}^{#2}}
\newcommand{\sbigcup}[2]{\bigcup\limits_{#1}^{#2}}

\newcommand{\lims}[1]{\lim\limits_{#1}}
\newcommand{\limsups}[1]{\limsup\limits_{#1}}
\newcommand{\liminfs}[1]{\liminf\limits_{#1}}
\newcommand{\sups}[1]{\sup\limits_{#1}}
\newcommand{\infs}[1]{\inf\limits_{#1}}
\newcommand{\bigcaps}[1]{\bigcap\limits_{#1}}
\newcommand{\bigcups}[1]{\bigcup\limits_{#1}}
\newcommand{\mins}[1]{\min\limits_{#1}}

\newcommand{\sums}[1]{\sum\limits_{#1}}
\newcommand{\ints}[1]{\int\limits_{#1}}

\newcommand{\sequence}[2]{\{#1\}_{#2}}

\newcommand{\Cb}{\bar{C}(\Reals^{d})}
\newcommand{\Czero}{C_{0}(\Reals^{d})}
\newcommand{\testfcts}{C^{\infty}_{c}(\Rd)}
\newcommand{\BoundedMeasurable}{B(\mathbb{R}^{d})}

\newcommand{\ProbsOnRd}{\mathcal{P}(\Reals^{d})}

\newcommand{\SkoSpace}{D_{\Reals^{d}}([0,\infty))}
\newcommand{\SkoSpaceOne}{D_{\Reals}([0,\infty))}

\newcommand{\Indicator}{\mathbf{1}} 
\newcommand{\IndOne}{\Indicator_{I_{1}}} 
\newcommand{\IndTwo}{\Indicator_{I_{2}}} 

\newcommand{\gt}{\rightarrow} 
\newcommand{\gtinf}{\rightarrow \infty} 
\newcommand{\Reals}{\mathbb{R}} 
\newcommand{\C}{\mathbb{C}} 
\newcommand{\Rd}{\Reals^{d}} 
\newcommand{\Def}{\mathcal{D}} 
\newcommand{\equaldist}{\stackrel{\mathcal{D}}{=}} 
\newcommand{\infinity}{\infty} 
\newcommand{\x}{\times} 
\newcommand{\scal}[2]{\langle#1,#2\rangle} 

\renewcommand{\Re}{\operatorname{Re}} 

\numberwithin{equation}{section}
\numberwithin{thm}{section}
\title{Solutions of martingale problems for L\'evy-type operators and stochastic
differential equations driven by L\'evy processes with discontinuous coefficients}

\author{Peter Imkeller \footnote{Institut f\"ur Mathematik,
Humboldt-Universit\"at zu Berlin, Unter den Linden 6, 10099 Berlin, Germany}  \and
Niklas Willrich \footnote{Weierstrass-Institut Berlin, Mohrenstr. 39, 10117
Berlin, Germany, \href{mailto:willrich@wias-berlin.de}{willrich@wias-berlin.de}}\\
}
\date{\today}

\begin{document}
\maketitle

\begin{abstract}
We show the existence of L\'evy-type stochastic processes in one space dimension with
characteristic triplets that are either discontinuous at thresholds, or are stable-like
with stability index functions for which the closures of the discontinuity sets are
countable. For this purpose, we formulate the problem in terms of a Skorokhod-space
martingale problem associated with non-local operators with discontinuous coefficients.
These operators are approximated along a sequence of smooth non-local operators giving
rise to Feller processes with uniformly controlled symbols. They converge uniformly
outside of increasingly smaller neighborhoods of a Lebesgue nullset on which the
singularities of the limit operator are located.
\end{abstract}

{\bf 2000 AMS subject classifications:} 60J25, 60H10, 60J75, 60G46, 60G52, 47G30.

{\bf Key words and phrases:} L\'evy process, stable process, stable-like
process,L\'evy-type process, discontinuous L\'evy characteristics, non-local operator,
Markov process, Feller process, symbol, martingale problem, weak solution, stochastic
differential equation, Skorokhod space, pseudo-differential operator.

\section{Introduction}

L\'evy processes, in particular processes with non-Gaussian
characteristics, have been playing an increasingly important role in the
modeling of natural phenomena in the recent decades. For instance,
Ditlevsen \cite{Ditlevsen1} conducts a study of Greenland ice core data
from the GRIP project which provide temperature proxies from a period of
about 100 T years covering roughly the last terrestrial glacial period. In
an attempt to find the best fitting one among models of dynamical systems
perturbed by symmetric $\alpha$-stable noise, he came up with a calibration
in which $\alpha\sim 1.75$ provides the best interpretation of the data
among these models. A look at the time series, however, reveals that
symmetric $\alpha$-stable noise is unrealistic as a modeling hypothesis,
since jump characteristics seem to depend strongly on the positions from
which jumps originate. This suggests that a more ample class of dynamical
systems perturbed by \emph{stable-like L\'evy noise}, i.e. with a space
($x$)-dependent stability index function $\alpha(x)$ might provide a better
modeling background. It has even been suggested to think about thresholds
$x_0$, below which roughly the noise acting on the system shows stable
behavior with stability index $\alpha_1$, while above the threshold a
possibly different stability index $\alpha_2$ is observed. Similar
questions have been asked in the context of the modeling of noise in
hydrodynamical models for diffusion of water in porous media in which
layers of different materials meet discontinuously at critical thresholds.

Already on the level of the modeling of noise processes in space dimension one, these
suggestions lead to considerable mathematical challenges. They trigger questions as: how
to construct stochastic processes behaving like stable processes with different stability
indices on the two sides of a threshold? This question was at the origin of the study
conducted in this paper. The problem we will solve here may be posed in its most general
form for one-dimensional stochastic processes in the following way: given a triple $(a, b,
\nu)$ with measurable functions $a, b:\mathbb{R}\to\mathbb{R}$ with $a>0$ and a Borel
measure valued measurable function $\mathbb{R}\ni x\mapsto \nu(x,\cdot)$; under which
conditions on the smoothness of $(a,b,\nu)$ will the triple be the characteristic triplet
of a L\'evy-type process?

The existence of a L\'evy-type stochastic process with characteristics of this type may,
as usual in the theory of stochastic processes, be proved in different ways. The problem
may be posed in terms of a stochastic differential equation in which the driving noise
reflects the assumptions on $a, b$ and $\nu$. Then one can proceed to show the existence
of strong solutions, or more generally of weak solutions. One may, given the properties of
$a, b$ and $\nu$, define the associated unbounded non-local operator $A$ and look for solutions of the the martingale problem on
appropriate Skorokhod spaces of c\`adl\`ag functions on $\mathbb{R}_+$. One may finally
investigate questions related to the links between weak solutions and the solutions of the
martingale problem. Our study below will be focused on the field of problems thus
circumscribed. There has been some work in the literature dedicated to various aspects
thereof. One may consult Bass \cite{BassSDE} for an overview of solutions of stochastic
differential equations driven by L\'evy processes with characteristic triplets with some
relaxed smoothness, and a rich reference list. Bass \cite{BassJumpProcesses} tackles
problems of existence and uniqueness of L\'evy-type processes from the point of view of
stochastic calculus and the theory of Markov processes, while Negoro
\cite{NegoroTransition} favors the approach via integral transforms and symbolic calculus
with a more analytical touch. Tsuchiya \cite{TsuchiyaLevy} treats stable-like processes.
Komatsu \cite{Komatsu} and
Stroock \cite{StroockLevyType} show the existence of solutions for martingale problems of
L\'evy-type operators, and investigate their well-posedness, either under the existence of
a smoothing Brownian component, or other additional continuity hypotheses. Boettcher
\cite{BoettcherTransience} uses tools of symbolic calculus to treat properties of
L\'evy-type processes that are stable-like with discontinuous coefficients, while assuming
(but not proving) the existence of strong Markov processes with these characteristics.
Fukushima and Uemura \cite{FukuUemuraHunt} show how to approach
L\'evy-type operators by means of lower-bounded semi-Dirichlet forms, and this way obtain
stable-like processes with coefficients that are smooth enough (see also
Schilling and Wang \cite{SchillingSemiDirichlet}).

In our approach the characterization of Feller processes via generalizations of the
L\'evy-Khinchin formula plays a crucial role. They are based on the theory of
pseudo-differential operators (see Jacob \cite{jacob1}, \cite{jacob2}, \cite{jacob3}). In
effect, they allow a representation of the non-local operators associated with Feller
processes by a type of Fourier inversion formula in which the generalization of the
characteristic exponent $q(\xi), \xi\in\mathbb{R},$ of the L\'evy-Khinchin formula, the
so-called \emph{symbol} $q(x,\xi), \xi, x\in\mathbb{R}$, appears. It can be seen as a
space ($x$)-dependent version of the exponent in the Fourier transform of the
corresponding Feller process. Symbols carry important information about the related
stochastic dynamics. In Schilling and Wang \cite{SchillingWangTransition} their growth
properties as functions of $\xi$, uniform in $x$, have been used to show the existence of
transition densities for the associated Feller dynamics, and are seen to give upper
estimates for their growth behavior. Our method of construction of the L\'evy-type
processes solving at least the martingale problem for the associated operator $A$ starts
in a general framework. It is later shown to agree with both the scenario of L\'evy
triplets that have a discontinuity at a fixed threshold, as well as with the one of purely
stable-like behavior with a discontinuous stability index function. In our main
approximation result Theorem \ref{thm:centralresult} we approach $A$ by a sequence of
operators $(A_n)_{n\in\mathbb{N}}$ which possess smooth coefficient triplets, and
corresponding symbols $(q_n)_{n\in\mathbb{N}}$. As an essential ingredient of our method,
we have to control the sequence of symbols \emph{uniformly in $n$} in a similar manner as
Schilling and Wang \cite{SchillingWangTransition} control one symbol. In consequence we
first obtain tightness of the unique solutions of the martingale problems
$(\mathbb{P}_n)_{n\in\mathbb{N}}$ associated with $(A_n)_{n\in\mathbb{N}}$. The
approximation further has to guarantee that the singularities of $A$ are contained in
$\cap_{m\in\mathbb{N}} U_m$, with a decreasing sequence $(U_m)_{m\in\mathbb{N}}$ of sets
with asymptotically vanishing Lebesgue measure, and that the $A_n$ converge to $A$
compactly uniformly on the complements of the $U_m$. To deduce that a cluster point
$\mathbb{P}$ of the sequence $(\mathbb{P}_n)_{n\in\mathbb{N}}$ solves the martingale
problem as well, on $U_m$ one uses once more the uniform  bounds on the transition
densities resulting from the control on the symbols. The construction of the appropriate
sequences $(A_n)_{n\in\mathbb{N}}$ and $(U_m)_{m\in\mathbb{N}}$ in the two scenarios
mentioned can be achieved. In case of the L\'evy characteristics that are discontinuous at
a threshold, this is done in Theorem \ref{thm:martingaleprobglueing}. In this case,
techniques developed in Kurtz \cite{EthierKurtz} further reveal that the solution
of the martingale problem also gives rise to a weak solution of the corresponding
stochastic differential equation (Theorem \ref{thm:weakexistencelevy}). In case the
closure of the set of discontinuities of the stability index function $\alpha$ is
countable,
Theorem \ref{thm:stablelikeapplication} provides the solution of the martingale problem.
Results on uniqueness of solutions or well-posedness of the martingale problems
considered, as well as the existence of strong solutions to stochastic differential
equations studied have to be left to future work.

The material of the paper is organized as follows. Section
\ref{sec:theoretischerhintergrund}
 contains summaries of the essentials about central
concepts on function spaces, Markov processes together with their semigroups, generators
and symbols, martingale problems on Skorokhod function spaces together with compactness
criteria on spaces of probability measures. We recall essentials about L\'evy processes,
in particular stable-like processes, and give an account of the relationship between
solutions to martingale problems and weak solutions of stochastic differential equations.
In section \ref{sec:existenzsatz} we provide the general existence theory for martingale
problems associated with singular operators. In section \ref{sec:anwendung} finally, the
general
existence theorem is applied to the scenarios of discontinuous L\'evy characteristics at a
threshold in space, and to stability index functions with discontinuity sets possessing
countable closures, to show existence of solutions for the martingale problem and weak
solutions of related stochastic differential equations.

\section{Theoretical background}\label{sec:theoretischerhintergrund}

\subsection{Function spaces, norms and notations}
If for $d \in \mathbb{N}$ we deal with Euclidean space $\Rd$ or $\mathbb{C}$, then
$\scal{\cdot}{\cdot}$ denotes the scalar product and $\|\cdot \|$ the Euclidean
norm. For the open ball of radius $r >0$ centered at $x \in \mathbb{R}^d$ we write
$B_{r}^{d}(x)$. We
omit the superscript for the
dimension, if it is clear from the context. The closure of a set $A$ in a
topological space will be written $\overline{A}$, and $A^{c}$ will denote its
complement. On function spaces $\|\cdot\|_{\infty}$ will be the uniform norm.
We consider the following subspaces of $M(\mathbb{R}^d)$, the space of Borel measurable
functions from $\Rd$ to $\Reals$. $\BoundedMeasurable$ is the space of bounded measurable
functions. By $\bar{C}(\Rd)$ we denote the space of bounded continuous functions.
More generally, for every metric space $E$, we write $\bar{C}(E)$ for the space of bounded
continuous functions
from $E$ to $\Reals$.
By $C^{k}(\Rd)$, $k \in \mathbb{N} \cup \{\infty\}$, we denote the space of functions
with continuous partial derivatives of order up to $k$, and by $C_{c}^{k}(\Rd)$ the linear
subspace of those functions that in addition have compact support.  $C_{c}^{\infty}(\Rd)$
will also be called the space of test functions.

$C_{0}(\Rd)$ symbolizes the space of continuous functions satisfying
\begin{equation}
 \lims{\|x\|\gt \infty}f(x)=0,
\end{equation}
while $\overline{C}^{1}(\Rd)$ stands for the space of bounded continuous functions with bounded
continuous
partial derivatives of first order.

For $d \in
\mathbb{N}$ we call an element $\alpha=(\alpha_{1},\ldots,\alpha_{d}) \in \mathbb{N}_{0}^{d}$
a multi-index and we introduce the following notation:
for $x=(x_{1},\ldots, x_{d}) \in \Rd$
and $\alpha \in \mathbb{N}_{0}^{d}$, we write
\begin{equation}
 x^{\alpha}:=x_{1}^{\alpha_{1}}\cdot \ldots \cdot x_{d}^{\alpha_{d}},
\end{equation}
and analogously for $\alpha \in \mathbb{N}_{0}^{d}$ and the composition of partial
derivatives $\partial_{x_{i}}:=\frac{\partial}{\partial x_{i}}$:
\begin{equation}
 \partial^{\alpha}:=\partial_{x_{1}}^{\alpha_{1}} \ldots
\partial_{x_{d}}^{\alpha_{d}}.
\end{equation}
We introduce the Schwartz space $\mathcal{S}(\Rd)$ by
\begin{equation}
 \mathcal{S}(\Rd):= \{ f \in C^{\infty}(\Rd) ~|~ \forall \alpha,\beta
\in \mathbb{N}_{0}^{d} ~:~ \sups{x \in \Rd}
|x^{\alpha}\partial^{\beta}f(x)|<\infty
\}.
\end{equation}
The Borel sets of a metric space $E$ are written $\mathcal{B}(E)$. For a set $A$ we
will denote by $\Indicator_{A}$ the indicator function of $A$. For
equality in distribution we will use the symbol $\equaldist$. For the induced distribution of a random
variable $X$ on a probability space $(\Omega,\Filt, \mathbb{P})$ taking values in a
measurable space $(E,\mathcal{E})$ we will write $\mathbb{P}^{X}(A):=\mathbb{P}\left( X
\in A \right), A \in \mathcal{E}$. For basic probability results like the Portmanteau
lemma or the continuous mapping theorem we refer to \cite{Kallenberg}. We will write
$\SkoSpace$ for the space of c\`adl\`ag functions from $\Rd$ to $\Reals$ equipped with
the Skorokhod topology (for a definition and properties of this space see chapter
3,\cite{EthierKurtz}).
If not stated otherwise, we will assume that a stochastic process takes its values in $\Rd$
for some $d\in\mathbb{N}$.

\subsection{Criteria for relative compactness}
In the following we consider the canonical process $X=(X_t)_{t\geq0}$ on $\SkoSpace$ and
a sequence of probability measures $\sequence{\mathbb{P}^{n}}{n\geq1}
\subset \mathcal{P}(\SkoSpace)$.
\begin{prop} \label{prop:relcomp1}
Let $(\Filt_{t})_{t \geq0}$ denote the canonical
filtration of $(X_t)_{t\geq0}$. The sequence $\sequence{\mathbb{P}^n}{n\geq1}$ is
relatively compact, if for all $t>0$:
\begin{equation}\label{RK1}
 \lims{K' \gtinf}\sups{n \geq 1} \Probn{\sups{s\leq t} \|X_s\| > K'}=0,
\tag{RK1}
\end{equation}
and for all $T>0$ there exists a family of non-negative random
variables $\sequence{\gamma_{n}(\delta)}{n\geq 1; 0<\delta< 1}$ such that
\begin{equation}\label{RK2}
 \En{\|X_{t+s}-X_t\|^{2} |\Filt_{t}}\leq
\En{\gamma_{n}(\delta)|\Filt_{t}} \tag{RK2}
\end{equation}
for $0 \leq t \leq T$, $0 \leq s \leq \delta$,
and
\begin{equation}\label{RK3}
 \lims{\delta \gt 0} \sups{n \geq 1}\En{\gamma_{n}(\delta)}=0. \tag{RK3}
\end{equation}
\end{prop}
\begin{proof}
This is a special case of theorem 8.6 with remark 8.7, chapter 3 in
\cite{EthierKurtz}.
\end{proof}
The following result relates the relative compactness of sequences of distributions
$\sequence{\mathbb{P}^n}{n\geq1} \subset \mathcal{P}(\SkoSpace)$ to the relative
compactness of sequences of distributions $\sequence{(\mathbb{P}^n\circ f(X)^{-1})}{n\geq1}$
for $f \in \testfcts$, where $\mathbb{P}\circ Y^{-1}$ denotes the law of the random map $Y$ under the probability measure $\mathbb{P}.$
\begin{prop}\label{prop:relcomp2}
Let $\sequence{\mathbb{P}^{n}}{n\geq1}$ be a sequence of probability measures on $\SkoSpace$. Assume condition \eqref{RK1} is satisfied. Then
$\sequence{\mathbb{P}^{n}}{n\geq1}$ is relatively compact if and only if for all $f \in
\testfcts$ the sequence $\sequence{(\mathbb{P}^n\circ f(X)^{-1})}{n\geq1}$ is relatively compact.
\end{prop}
\begin{proof}
See Theorem 9.1, chapter 3, \cite{EthierKurtz}.
\end{proof}

\subsection{Markov processes, semigroups, generators and symbols}
We give a quick reminder of the concept of semigroups arising in the study of Markov processes.
Let $(L,\|\cdot\|_L)$ be a Banach space and $L \subset \BoundedMeasurable$.
\subsubsection{Semigroups of operators}
A family of bounded operators $(T_{t})_{t \geq 0}$ on  $L$ is said to be a \emph{semigroup} if
it satisfies for all $t,s \geq 0$ the conditions $T_{0}=\operatorname{id}_{L}$ and
$T_{t+s}=T_{t}\circ T_{s}$, where ${id}_{L}$ is the identity operator on $L$.
We say that $(T_{t})_{t \geq 0}$ is \emph{strongly continuous}, if $\lims{t\gt
0}T_{t}f=f$ for all $f \in L$. $(T_{t})_{t \geq 0}$ will be called a \emph{contraction
semigroup} if $\|T_{t}\|_{L} \leq 1$ for all $t \geq 0$ and \emph{positive} if for all $t \geq 0$
and $ f \geq 0$ the inequality $T_{t}f\geq 0$ holds.
\begin{definition}
A positive strongly continuous contraction semigroup $(T_{t})_{t\geq 0}$ on
$(\Czero,\|\cdot\|_{\infty})$ will be called a \emph{Feller semigroup}.
\end{definition}
A Feller semigroup can be uniquely extended to a semigroup $(T_{t})_{t\geq0}$ on
$\BoundedMeasurable$. If $T_{t}1=1$ for all $t \geq 0$, then $(T_{t})_{t\geq
0}$ is \emph{conservative}.\\
In the following we will write $\Def(A)$ for the domain of a possibly
unbounded operator $A$ between normed spaces and we will sometimes write the operator
as $(A,\Def(A))$.
Further we denote by $\mathcal{R}(A)$ the range of $A$. We say that $A$ is an operator
 \emph{on} a Banach space $L$, if $\Def(A)$,$\mathcal{R}(A) \subset L$.
Let $(A,\Def(A))$ be an operator on $L$ and $G \subset \Def(A)$. The \emph{restriction}
$A|_{G}$ of $A$ on $G$ is defined as $A|_{G}f:=Af, f \in G$ and
$\Def(A|_{G}):=G$.
Let $(A_{1},\Def(A_{1})),(A_{2},\Def(A_{2}))$ be two operators on the Banach space
$L$. If $\Def(A_{1}) \subset \Def(A_{2})$
and $A_{2}|_{\Def(A_{1})}=A_{1}$, we say that $A_{2}$ is an \emph{extension} of $A_{1}$
on $L$.\\
For an operator $A$ on $L$ we denote by $\mathcal{G}(A):=\{(f,Af) \in L
\times
L ~|~ f \in \Def(A)\}$ the \emph{graph} of $A$. If $\mathcal{G}(A)$ is closed in the
product topology of $L \times L$, we say that $A$ is closed. The smallest closed
extension of an operator $A$ is called its closure $\overline{A}$.
\begin{definition}
Let $(T_{t})_{t \geq 0}$ be a semigroup on $L$. We define its (infinitesimal) generator
$A$ as follows:
\begin{equation}
  Af := \lims{t \gt 0} \frac{T_{t}f-f}{t}
\end{equation}
for all $f \in \Def(A)$. Here $\Def(A)$ is the set of all $f \in L$ for which the limit exists in a strong sense.
\end{definition}

\subsubsection{Markov processes}

We will work with the concept of \emph{universal Markov processes} (see \cite{jacob3}). In the
study of universal Markov processes a semigroup appears in a natural way.
For $t \geq 0$ we consider
\begin{equation} \label{eq:semigroup}
 T_{t}f(x):=\Ex{f(X_t)}, x \in \Reals^{d}, \quad \text{for all} ~f \in
\BoundedMeasurable.
\end{equation}
One can show that $(T_{t})_{t \geq 0}$ is a positive contraction semigroup for every
universal Markov process $(X_t)_{t\geq0}$. In certain cases one can also construct,
starting from a given semigroup, a universal Markov process with this particular
semigroup.
\begin{thm}\label{thm:existstochastiprocesses}
Let  $(T_{t})_{t\geq }$ be a conservative Feller semigroup. For every $\mu \in
                     \ProbsOnRd$, there exists a stochastic
process $(\Omega,\mathcal{F},\mathbb{P}^{\mu},(X_t)_{t\geq0})$ such that
\begin{equation}
\mathbb{P}^{\mu}(X_{0} \in \Gamma)=\mu(\Gamma) \quad \text{for all }~\Gamma
\in \mathcal{B}(\Rd)
\end{equation}
and
\begin{equation}
 T_{s}f(X_t)=\mathbb{E}^{\mu}\left({f(X_{t+s})|X_t}\right) \quad \text{for
all}~ f \in
\Czero.
\end{equation}
Such a process will be called \emph{Feller process}. Its finite dimensional
distributions are uniquely defined by the semigroup $(T_{t})_{t \geq 0}$ and $\mu$.
\end{thm}
\begin{proof}
See theorem 2.7, chapter 4, \cite{EthierKurtz} and theorem 1.6, chapter 4,
\cite{EthierKurtz}.
\end{proof}
Furthermore the Feller processes associated with a given Feller semigroup and the
initial distributions $\delta_{x}, x \in \Reals,$ form a universal Markov process
(corollary 3.4.10 \cite{jacob3}). One can always assume that a modification of the Feller
process with paths in $\SkoSpace$ is given (theorem 3.4.19,\cite{jacob3}).
This means that we can work in the setting of a canonical process on
$\SkoSpace$ to study Feller processes.\\

\subsubsection{Symbols of Feller processes}\label{sec:GeneratorSymbol}
A useful approach to the study of the generator in case it is a
pseudo-differential operator is to study its symbol.
Let $A$ be a generator of a Feller process with $C_{c}^{\infty}(\Rd)
\subset \Def(A).$
Then the restriction of $A$ on $\testfcts$  is a pseudo-differential operator, given by
\begin{equation}
 Af(x)=-\int e^{i\scal{x}{\xi}}q(x,\xi)\hat{f}(\xi)d\xi
\quad \text{for} ~
f \in \testfcts,
\end{equation}
where $\hat{f}$ denotes the Fourier transform of $f$ described by
$\hat{f}(\xi):=\frac{1}{(2\pi)^d}\int e^{-i\scal{x}{\xi}}f(x)dx, \xi \in \Rd$.
The function $q(\cdot,\cdot):\Rd \x \Rd \rightarrow \C$ is called the symbol of the
operator $A$.\\
The use of the symbol in the study of Markov processes has been an active area of research
in the past years. A detailed exposition of the approach can be found in
\cite{jacob1},\cite{jacob2},\cite{jacob3}.
Especially in the case of Feller processes the approach seems to be a very efficient one.
In the following we summarize some results, which we will use in our proofs.
We start by stating a bound on the maximal deviation of a Feller process from its starting
point before time $t$.
\begin{prop}\label{prop:exittime}
Let $((X_t)_{t\geq0},\mathbb{P}^{x})$ be a Feller process in $\Rd$
with generator $(A,\Def(A))$ and starting point $x \in \Rd$, such that $\testfcts
\subset \Def(A)$.  We denote its symbol by $q$. Assume there exists $C>0$ such that for
all $\xi \in \Reals^d$
\begin{equation}
\begin{split}
\|q(\cdot,\xi)\|_{\infty}   \leq &\,  C(1+\|\xi\|^{2}),\\
q(\cdot,0)  = &\,  0.
\end{split}
\end{equation}
Then for all $K'>0$ and $t \geq0$ we have
\begin{equation}
 \Probx{\sups{s\leq t} \|X_s-x\| \geq K'} \leq Ct\sups{\xi\leq 1/K'}
\|q(\cdot,\xi)\|_{\infty}.
\end{equation}
\end{prop}
\begin{proof}
See Proposition 4.3,\cite{SchillingWangTransition}.
\end{proof}
The next result will be very important for our proofs, since it allows us to show the
existence of a transition density of a Feller process and bound it explicitly in
terms of the symbol.
\begin{thm}\label{thm:TransitionDensities}
 Let $((X_t)_{t\geq0},\mathbb{P})$ be a Feller process with generator $(A,\Def(A))$,
such that $\testfcts \subset \Def(A)$. Then
$A|_{\testfcts} = -
q( \cdot,D)$
is a pseudo-differential operator with symbol $q$. Assume that $q$ satisfies the
properties: For some $C>0$ and for all $\xi \in \Reals^d$
\begin{equation}
 \| q(\cdot,\xi)\|_{\infty} \leq C(1+\|\xi\|^{2}) \quad \text{and} \quad
q(\cdot,0) =0.
\end{equation}
If moreover
\begin{equation}
 \lims{\|\xi \|\gtinf} \frac{\infs{z \in \Reals} \Re
q(z,\xi)}{\log(1+\|\xi\|)}= \infty,
\end{equation}
then the process $(X_t)_{t \geq 0}$ has a transition density $p(t,x,y),t \in
[0,\infty), x,y \in \Rd,$ with respect to Lebesgue measure and the following inequality holds for $t>0$:
\begin{equation}
 \sups{x,y \in \Rd} p(t,x,y) \leq \frac{1}{(4\pi)^{d}} \int \exp\left(
-\frac{t}{16}\infs{z\in \Rd} \Re q(z,\xi) \right) d\xi.
\end{equation}
\end{thm}
\begin{proof}
See Theorem 1.1,\cite{SchillingWangTransition}.
\end{proof}
In case the canonical process $(X_t)_{t\geq0}$ is a L\'evy process under $\mathbb{P}$,
the symbol is
independent of $x$ and the conditions reduce to
\begin{equation}\label{eq:hartmannwintner}
  \lims{\|\xi \|\gtinf} \frac{\Re q(\xi)}{\log(1+\|\xi\|)}= \infty.
\end{equation}
We will call this condition the \emph{Hartman-Wintner condition}. In
\cite{hartmannwintner} Hartman and Wintner showed that this condition is sufficient for
a L\'evy process to possess a continuous transition density in $C_0(\Rd)$ (a discussion
and an extension of this result can also be found in \cite{knopova1}).\\
We will also use the following theorem, which provides us with a way to construct Feller
processes starting with symbols of L\'evy processes (continuous negative definite
functions with $c=0$). We denote by $M^T$ the transposition of a matrix $M$.
\begin{thm}\label{thm:glueing}
Let $q:\Reals^d \times \Reals^d \gt \C$ be a function of the following structure:
\begin{equation}
 q(x,\xi)= \psi (\phi^{T}(x)\xi),x \in \Rd, \xi \in \Rd,
\end{equation}
where $\psi:\Reals^{n} \gt \C$ is a continuous negative-definite function with $c=0$
and
$\phi:\Rd \gt \Reals^{d \times n}$ is bounded and Lipschitz-continuous. Then there exists
a probability measure $\mathbb{P}$ under which the canonical process is Feller, such that
the generator $(A,\Def(A))$ satisfies
$C^{\infty}_{c}(\Reals^{d}) \subset
\Def(A)$ and $A$ has the symbol $q$.
\end{thm}
\begin{proof}
This is Corollary 3.7,\cite{SchillingSchnurrSDE}.
\end{proof}
We finish this section by pointing out that if the canonical process is Feller under
$\mathbb{P}$, such that the conditions of theorem
\ref{thm:TransitionDensities} are satisfied, it is a \emph{strong} Feller process, i. e. its semigroup
maps bounded measurable functions to continuous bounded functions. In this case, one can
write the symbol also as (see section 3.9, \cite{jacob3})
\begin{equation}
 q(x,\xi)=-\lims{t \gt 0} \frac{\Ex{e^{i\scal{X_t-x}{\xi}}}-1}{t}.
\end{equation}
This gives a probabilistic interpretation of the symbol (so far defined entirely in
analytic terms).

\subsubsection{L\'evy processes} \label{sec:levyprocesses}
L\'evy processes can be identified by the characteristic functions of their
increments via the L\'evy-Khinchine formula.
To briefly recall its main features, we first introduce the concept of a continuous negative-definite function.
\begin{definition}
A function $q:\Rd \gt \C$ is continuous negative-definite if it can be written in the
form
\begin{equation}
q(\xi)=c-i\scal{b}{\xi}+\frac{1}{2}\scal{\xi}{a\xi}-\ints{\Rd\backslash\{0\}}
(e^
{i\scal{\xi}{y}}-1-i\scal{\xi}{y}\Indicator_{\overline{B^{d}_{1}(0)}}(y))\nu(dy)
\end{equation}
with $c \geq 0$, $b \in \Rd$, $a \in \Reals^{d \x
d}$ a positive-semidefinite symmetric matrix and $\nu$ a measure on
$\mathcal{B}((\Rd\backslash\{0\})$ satisfying
\begin{equation}
 \ints{\Rd\backslash\{0\}}(1\wedge \|y\|^{2})\nu(dy)<\infinity.
\end{equation}
This measure will be called \emph{L\'evy measure}.
\end{definition}
One can extend the definition of $\nu$ to the whole of $\Rd$ by assuming
$\nu(\{0\})=0$. It is known that every L\'evy measure is $\sigma$-finite.\\
In the following lemma we summarize some properties of continuous negative-definite
functions from the literature.
\begin{lemma} Let $q$ be a continuous negative-definite function. We have
\begin{itemize}
\item $q$ is continuous,
\item there exists $C >0$, such that for all $\xi \in
\mathbb{R}^{d}$: $\|q(\xi)\| \leq C(1+\|\xi\|^{2})$,
\item $q(0)=0$ if $c=0$,
\item $ \Re q(\xi)\geq 0 ~\text{for all } \xi \in \Rd$.
\end{itemize}
\end{lemma}
\begin{proof}
The first property follows from Theorem 3.7.7 in \cite{jacob1}, the second from Lemma 3.6.22, \cite{jacob1}, the third one is trivial, while property four
follows from $1-\cos(\scal{\xi}{y}) \geq 0$ for all $y,\xi \in
\Rd$.
\end{proof}
We now recall the L\'evy-Khinchine formula.
\begin{thm}[\textit{L\'evy-Khinchine formula}]
Let $(L_{t})_{t\geq0}$ be a L\'evy process taking values in $\Rd$. Then for all
$t\geq0$ the characteristic function of $L_{t}$ is of the form
\begin{equation}\label{eq:LevyKhinchin}
 \E{e^{i\scal{\xi}{L_{t}}}}=\exp(-tq(\xi)),
\end{equation}
where $q$  is a continuous negative-definite function with $c=0$. This function
uniquely determines the finite-dimensional distributions of $(L_{t})_{t\geq0}$. We call
$(b,a,\nu)$
the characteristic triple of the process.
\end{thm}
\begin{proof}
Theorem 1.2.14 and theorem 1.3.3 in \cite{Applebaum}, can be adapted to our notations.
\end{proof}
We call $q$ the \emph{characteristic exponent} of the L\'evy process.
A L\'evy process is a Feller process (see theorem 3.1.9 \cite{Applebaum}).
So every L\'evy process gives rise to a Feller semigroup. We are interested in the
structure of its generator and symbol. The following theorem provides some information.
\begin{thm}
Let $(L_{t})_{t\geq0}$ be a L\'evy process with characteristic exponent $q$ and
L\'evy triple
$(b,a,\nu)$. Let $(T_{t})_{t\geq0}$ be the associated Feller semigroup and $A$
its infinitesimal generator.
\begin{enumerate}
 \item For every $f \in \mathcal{S}(\Rd)$, $x \in \Rd$, we have
\begin{equation}
Af(x)=-\ints{\Rd}e^{i\scal{\xi}{x}}q(\xi)\hat{f}(\xi)d\xi.
\end{equation}
Hence $A$ is a pseudo-differential operator with symbol $q$.
\item For every $f \in \mathcal{S}(\Rd)$, $x \in \Rd$ we have
\begin{equation}
\begin{split}
 Af(x)= &\,
\scal{b}{\nabla
f(x)}+\frac{1}{2}\sums{1 \leq i,j \leq d}a_{i,j}\partial_{i}\partial_{j}f(x)\\
&\,
+ \ints{\Rd\backslash\{0\}}(f(x+y)-f(x)-\scal{y}{\nabla f(x)}\Indicator_{
\overline{B^{d}_{1 }(0)}}(y))\nu(dy).
\end{split}
\end{equation}
\end{enumerate}
\end{thm}
\begin{proof}
This is included in Theorem 3.3.3, \cite{Applebaum}.
\end{proof}
Another representation of L\'evy processes is contained in the L\'evy-It\^o decomposition.
Here
we use concepts from the Appendix  \ref{sec:poissonpoint}. Let $\lambda$ denote the Lebesgue
measure on $\Rd$.
\begin{thm}\label{thm:levyito}
Every L\'evy process $(L_{t})_{t\geq0}$ with L\'evy triple $(b,a,\nu)$ can be written in
the
following way:
\begin{equation} \label{eq:levyito}
 L_{t}=bt+B^{a}_{t}+\sint{0}{t}\ints{\|x\|\geq 1}
x\xi(ds,dx)+\sint{0}{t}\ints{\|x\|<
1} x\tilde{\xi}(ds,dx),t \geq 0,
\end{equation}
where $(B^{a}_{t})_{t\geq0}$ is a $d$-dimensional Brownian motion with covariance
matrix $a$, $\xi$ is an independent Poisson random measure on $[0,\infty) \times \Rd$ with
intensity measure $\lambda \otimes \nu$ and $\tilde{\xi}$ is the corresponding
compensated Poisson random measure. Conversely let a $d$-dimensional Brownian
motion $(B^{a}_{t})_{t\geq0}$ with covariance matrix $a$, an independent Poisson random
measure $\xi$ on $[0,\infty) \times \Rd$ with
intensity measure $\lambda \otimes \nu$ and corresponding
compensated Poisson random measure $\tilde{\xi}$ be given. Then \eqref{eq:levyito} defines a L\'evy
process with L\'evy triple $(b,a,\nu)$.
\end{thm}
\begin{proof}
This is Theorem 2.4.16 in \cite{Applebaum}.
\end{proof}
For $0<\alpha \leq 2 $ and $C>0$ we define the (one-dimensional) symmetric
$\alpha$-stable process as the L\'evy process with characteristic exponent
$C|\xi|^{\alpha}$.
Its generator can be written for an $h>0$ as
\begin{equation}
  A^{\alpha}f(x):=\ints{\Reals\backslash \{0
\}}(f(x+y)-f(x)-y\partial_xf(x)\Indicator_{[-1,1]\backslash\{0\}}(y))\frac{h}{|y|^{
1+\alpha}}dy, x \in \Reals,
\end{equation}
for $f \in C^{2}_{c}(\Reals)\subset \Def(A^{\alpha})$.
$C$ is then equal to
\begin{equation}
 C=\ints{\Reals\backslash\{ 0\}}(1-\cos(y))\frac{h}{|y|^{1+\alpha}}dy,
\end{equation}
and one can choose $h$, such that $C=1$.
In this case $q(\xi)=|\xi|^{\alpha}$, $\xi \in \Reals$.
\subsubsection{Martingale problems}
For a right-continuous Markov process $((X_t)_{t\geq0},\mathbb{P})$ with generator $A$
and $f\in \Def(A)$ the process
\begin{equation}
 \left(f(X_t)-f(X_{0})-\sint{0}{t}Af(X_s)ds\right)_{t\geq0}
\end{equation}
is a $\mathbb{P}$-martingale with respect to the canonical filtration of
 $(X_t)_{t\geq0}$ (see proposition 1.7, chapter 4,\cite{EthierKurtz}).
A question often posed in the literature is the following: can one find a
stochastic process, such that the above condition for a certain operator $(A,\Def(A))$ is
satisfied for all $f \in \Def(A)$?
This is formalized in the statement of the \emph{martingale problem}.
\begin{definition}
Let $(A,\Def(A))$ be an operator on $\BoundedMeasurable$ and $\nu \in
\ProbsOnRd$. We say that a stochastic process $(X_t)_{t\geq0}$ with
right-continuous paths is a solution of the martingale problem associated with $A$ and the probability measure $\mathbb{P}$, if
for all $f \in \Def(A)$
\begin{equation}\label{eq:mgproblem}
\left(f(X_t)-f(X_{0})-\sint{0}{t}Af(X_s)ds \right)_{t\geq0}
\end{equation}
is a $\mathbb{P}$-martingale with respect to the canonical filtration of
$(X_t)_{t\geq0}$ and $X_0$ has distribution $\nu$.
\end{definition}
If $(X_t)_{t \geq0}$ is the canonical process we will also say that $\mathbb{P}$ solves
the martingale problem.
We say that the process $(X_t)_{t\geq0}$ solves the
$\SkoSpace$-martingale problem with respect to a probability measure $\mathbb{P}$, if the paths of $(X_t)_{t\geq0}$ are in $\SkoSpace$.
The martingale problem is \emph{well-posed} in $\SkoSpace$, if for every initial
distribution there
exists a solution, which is unique in its finite-dimensional distributions.
In the case of well-posedness one can see that the solutions with respect to the initial
distributions $\delta_{x}$, $x \in \Rd,$ constitute a universal Markov process.\\

\subsection{Stable-like processes}\label{sec:stablelike}
Stable-like processes extend the class of symmetric $\alpha$-stable processes to space-dependent stability index functions.
The idea is to consider processes associated with the symbol
\begin{equation}
  q(x,\xi)=|\xi|^{\alpha(x)}, x \in \Reals, \xi \in \Reals.
\end{equation}
The generators of such processes can be expressed by
\begin{equation}
 A^{\alpha}f(x):=\ints{\Reals\backslash \{0
\}}(f(x+y)-f(x)-y\partial_xf(x)\Indicator_{[-1,1]\backslash\{0\}}(y))\frac{h(x)}{|y|^{
1+\alpha(x)}}dy, \quad x\in\Reals,
\end{equation}
with $\Def(A^{\alpha})=C^{2}_{c}(\Reals)$
and $\alpha\!:\!\Reals\gt (0,2)$ measurable.
It is possible to choose $h$ such that
\begin{equation}
 1=\ints{\Reals\backslash\{ 0\}}(1-\cos{y})\frac{h(x)}{|y|^{1+\alpha(x)}}dy.
\end{equation}
We see that  $q^{\alpha}$ is the symbol of $A^{\alpha}$. Existence and uniqueness of the
solution of a martingale problem associated with $A^{\alpha}$ have been treated in
\cite{BassJumpProcesses}.
\begin{thm}\label{thm:BassExistenceResult}
For a coefficient function $\alpha \in \bar{C}^{1}(\Reals)$ which satisfies
\begin{equation}
0 < \infs{x \in \Reals} \alpha(x) \leq \sups{x \in \Reals} \alpha(x)<
2,
\end{equation}
the $D_{\Reals}([0,\infty))$-martingale problem associated with $A^{\alpha}$
is well-posed and $\mathcal{R}(A^{\alpha}) \subset \bar{C}(\Reals)$.
\end{thm}
\begin{proof}
This is Corollary 2.3,\cite{BassJumpProcesses} and
Proposition 6.2,\cite{BassJumpProcesses}.
\end{proof}
Theorem 3.3 of \cite{SchillingWangTransition} extends this result by stating that the
solution of the
martingale problem is a Feller process and that the closure of $A^{\alpha}$ is the
generator of that Feller process.
\subsection{Equivalence of solutions of stochastic differential equations and martingale
problems}\label{sec:EquiStoDiff}
In this section we recall results from \cite{KurtzEquivalence} adapted to our needs.
We restrict ourselves to the situation of processes in $\Reals$, but the results can be
stated in $\Rd$ too. We denote by $\lambda$ the Lebesgue measure on $\mathbb{R}$.
We will use results from the Appendix \ref{sec:poissonpoint} in the following.
Let $(B_{t})_{t\geq0}=(B^{1}_{t},\ldots, B^{m}_{t})_{t\geq0}, m\geq1,$ be an
$m$-dimensional standard Brownian motion,
$(p_{t})_{t\geq0}$ a Poisson point process on the measurable space
$(S,\mathcal{S})$ with intensity measure $\mu$ independent of $(B_{t})_{t\geq0}$.
The Poisson random measure associated with $(p_{t})_{t\geq0}$ is denoted by $\xi$.
Let $f_i \in M(\Reals),1\leq i \leq m, g \in M(\Reals)$, $h$ be a measurable
function from $\Reals \times S$ to $\Reals$. Now we consider the stochastic
differential equation
 \begin{equation} \label{eq:GeneralSDE}
\begin{split}
 X_t= X_{0}+ \ssum{i=1}{m}\sint{0}{t}f_{i}(X_s)dB^{i}_{s} +
\sint{0}{t}g(X_s)ds\\
 + \sint{0}{t}\ints{\overline{B_{1}^{d}(0)}}
h(X_{s^{-}},y)\tilde{\xi}(ds,dy)\\
 + \sint{0}{t}\ints{\overline{B_{1}^{d}(0)}^{c}}
h(X_{s^{-}},y)\xi(ds,dy).
\end{split}
\end{equation}
Under the assumption that all terms are well-defined, we say that a \emph{weak} solution
of
this stochastic differential equation is the specification of a filtered probability
space $(\Omega,\mathcal{F},(\mathcal{F}_{t})_{t\geq0},\mathbb{P})$, an
$m$-dimensional standard Brownian motion $(B_{t})_{t\geq0}$,
a Poisson point process $(p_{t})_{t\geq0}$ on $S$ with intensity measure $\mu$ and
associated Poisson random measure $\xi$, which is independent of $(B_t)_{t\geq0}$, and
a stochastic process $(X_t)_{t\geq0}$, such that all three processes
are $(\Filt_{t})$-adapted and \eqref{eq:GeneralSDE} holds
$\mathbb{P}$-almost surely.
We often call the process $(X_t)_{t\geq0}$ by itself a weak solution of the SDE.
Now we consider operators of the form
\begin{equation}\label{eq:OperatorEqOriginal}
\begin{split}
 Af(x)  = &\, \frac{1}{2}a(x)\partial_{x}^{2}f(x)+b(x)\partial_{x}f(x)\\
 + &\, \ints{\Reals}
(f(x+y)-f(x)-y\partial_{x}f(x)\Indicator_{\overline{B}_{1}(0)}(y) )\eta(x,dy),
\quad x \in \Reals,
\end{split}
\end{equation}
for $f \in C_{c}^{\infty}(\Reals)$, where $b \in M(\Reals)$, $a \in M(\Reals)$, $a
\geq0,$ and for every $x \in \Reals$ $\eta(x,\cdot)$ is a L\'evy measure. Further we assume
that there exists a complete separable metric space $(S,d)$, a $\sigma$-finite
measure $\nu$ on $\mathcal{B}(S)$, the Borel-$\sigma$-algebra of $S$, and functions
$\zeta:\Reals \times S \gt \{0,1\}$ and $\gamma:S \gt \Reals$ such that
\begin{equation}
 \eta(x,\Gamma)=\ints{S} \zeta(x,u) \Indicator_{\Gamma}(\gamma(u)) \nu(du),
\Gamma \in
\mathcal{B}(\Reals), x \in \Reals.
\end{equation}
We also assume that $S$ can be written as a disjoint union of sets $S_{1}$ and $S_{2}$
which satisfy
\begin{equation}
\Indicator_{\overline{B}_{1}(0)}(\gamma(u))=\Indicator_{S_{1}}(u)
\end{equation}
and
\begin{equation}
 \ints{S_1} \zeta(x,u)\gamma(u)^{2}du+\ints{S_2} \zeta(x,u)\nu(du)
<\infty
\end{equation}
for all $x \in \Reals$. Then we can express the operator defined in
\eqref{eq:OperatorEqOriginal} by
\begin{equation}\label{eq:OperatorEquivalence}
\begin{split}
 Af(x)  = &\,\frac{1}{2}a(x)\partial_{x}^{2}f(x)+b(x)\partial_{x}f(x)\\
 + &\,\ints{S}\zeta(x,u)
(f(x+\gamma(u))-f(x)-\gamma(u)\partial_{x}f(x)\Indicator_{S_{1}}(u))\nu(du),
 \quad x \in \Reals.
\end{split}
\end{equation}
We are interested in the type of SDE
 \begin{equation} \label{eq:SpecialSDE}
\begin{split}
 X_t= X_{0}+ \ssum{i=1}{m}\sint{0}{t}\sigma_{i}(X_s)dB^{i}_{s} +
\sint{0}{t}b(X_s)ds\\
 + \sint{0}{t}\ints{S_1}
\zeta(X_{s^{-}},u)\gamma(u)\tilde{\xi}(ds,du)\\
 + \sint{0}{t}\ints{S_2}
\zeta(X_{s^{-}},u)\gamma(u)\xi(ds,du), m \geq1,
\end{split}
\end{equation}
where $\sigma=(\sigma_{1},\ldots, \sigma_{m})$ is chosen, such that
$a=\|\sigma\|^{2}$ and the Poisson random measure $\xi$ has intensity
measure $\lambda \otimes \nu$.
The following theorem states that a solution of a $\SkoSpaceOne$-martingale problem
associated with $A$ is equivalently a weak solution of the above stochastic differential
equation.
\begin{thm}\label{thm:KurtzEquivalence}
Let $A$ be an operator of the form \eqref{eq:OperatorEquivalence} with
$\Def(A)=C^{\infty}_{c}(\Reals)$ and $\mathcal{R}(A) \subset B(\Reals)$ such that for
every compact set $K\subset \Reals$ we have
\begin{equation}
 \sups{x \in K}\left(
|a(x)|+|b(x)|+
\ints{S_1}\zeta(x,u)\gamma(u)^2\nu(du)+\ints{S_2}\zeta(x,
u){|\gamma(u)|\wedge 1}\nu(du)\right) < \infty.
\end{equation}
Then every solution of the $\SkoSpaceOne$-martingale problem is also a weak solution
of the SDE \eqref{eq:SpecialSDE}.
\end{thm}
\begin{proof}
Theorem 2.3,\cite{KurtzEquivalence}, shows the statement for
$\Def(A)=C^{2}_{c}(\Reals)$. But an inspection of the proof shows that it is valid as well
for $\Def(A)=C^{\infty}_{c}(\Reals)$.
\end{proof}
\clearpage
\section{An existence theorem for martingale problems}\label{sec:existenzsatz}
In order to deduce the existence of a solution to related martingale problems, we will
need some relative compactness of the laws of Feller processes.
For this purpose it will be instrumental to have a criterion in terms of the symbols.
\subsection{Relative compactness of sequences of distributions of Feller processes}
\begin{prop}\label{prop:weakconvergence}
Let $\sequence{\mathbb{P}^{n,\nu}}{n\geq 1}$ be a sequence of
probability measures for which the canonical process is Feller with initial distribution $\nu \in \ProbsOnRd$ and corresponding generators
$\sequence{A_{n}}{n\geq1}$. Let $\testfcts \subset \Def (A_{n}) $ for all
$n\geq 1$, and suppose that the processes possess the symbols $\sequence{q^{n}}{n\geq 1}$.
Assume that the following conditions are satisfied for a $C>0$ and for all $\xi \in \Rd$:
\begin{align}
\label{eq:tightnesscondition0}
 q^{n}(\cdot,0)=0 \quad \text{for all} ~n \geq 1, \tag{A1} \\
\label{eq:tightnessconditions1}
 \sups{n \geq 1} \| q^{n}(\cdot,\xi)\|_{\infty} \leq
C(1+\|\xi\|^{2}),\tag{A2} \\
\label{eq:tightnessconditions2}
\lims{\xi' \gt 0} \sups{n \geq 1}\| q^{n}(\cdot,\xi')\|_{\infty}=0. \tag{A3}
\end{align}
Then the sequence of distributions $\sequence{\mathbb{P}^{n,\nu}}{n\geq1}$ is relatively
compact.
\end{prop}
\begin{proof}
For our proof we use proposition \ref{prop:relcomp2}.
First we show condition \eqref{RK1} with the help of proposition
\ref{prop:exittime}. In fact,
\begin{equation}
 \lims{K\gtinf}\limsups{n\gtinf} \ProbI{n,\nu}{\sups{s\leq t}\|X_s\| > K
}=0, t\geq0.\tag{RK1}
\end{equation}
This holds for  $t \geq 0$, because for $K>0$
\begin{align}
 &\, \ProbI{n,\nu}{\sups{s\leq t}\|X_s\| > K } \notag  \displaybreak[0]\\
\leq &\,  \ProbI{n,\nu}{\|X_{0}\| < K/2~;~\sups{s\leq t}
\|X_s\|>K} +
\ProbI{n,\nu}{\|X_{0}\| \geq K/2 } \notag \displaybreak[0]\\
\leq &\, \sups{\|x\| < K/2} \ProbI{n,x}{\sups{s\leq t}\|X_s\|>K} +
\ProbI{n,\nu}{\|X_{0}\| \geq K/2 } \notag\\
 \leq &\, \sups{\|x\| < K/2}\ProbI{n,x}{\sups{s\leq t}\|X_s-x\| >
(K-\|x\|)^{+}}+
\ProbI{n,\nu}{\|X_{0}\| \geq K/2 }\notag \displaybreak[0]\\
\leq &\, \sups{\|x\| < K/2}Ct\sups{\xi \leq 1/(K-\|x\|)^{+}
}\|q^{n}(\cdot,\xi)\|_{\infty}+
\ProbI{n,\nu}{\|X_{0}\| \geq K/2 } \notag\\
\leq &\, Ct\sups{\xi \leq
2/K}\|q^{n}(\cdot,\xi)\|_{\infty}+
\ProbI{n,\nu}{\|X_{0}\| \geq K/2 }.
\end{align}
Here we have used proposition \ref{prop:exittime} with $K'=(K-\|x\|)^{+}$.
Furthermore we see by condition \eqref{eq:tightnessconditions2}, that the first term of
the last line tends to zero uniformly in $n$ for $K \gtinf$. Finally, we use the fact that
the second term also tends to zero uniformly in $n$ for $K \gtinf$, since the initial distribution is common to all probability measures. So we have proved condition \eqref{RK1}.\\
To show relative compactness of the distributions
$\sequence{(\mathbb{P}^{n,\nu}\circ f(X)^{-1})}{n\geq 1}$ for a fixed $f \in
C_{c}^{\infty}(\Rd)$ we use proposition \ref{prop:relcomp1}. Condition
\eqref{RK1} is trivially satisfied by boundedness of $f$.  Now we prove
conditions \eqref{RK2} and \eqref{RK3}.\\
For $n\geq1$ we write $(\Filt_{t})_{t\geq0}$ for the canonical filtration of
$(X_t)_{t\geq0}$.
For $0<s<\delta<1$ and $n \geq 1$ we have
\begin{equation}
\begin{split}
  &\, \EI{n,\nu}{(f(X_{t+s})-f(X_t))^{2} | \Filt_{t}}\\
= &\, \EI{n,\nu}{f^{2}(X_{t+s})-f^{2}(X_t)|\Filt_{t}} \\
 &\, - 2f(X_t)\EI{n,\nu}{f(X_{t+s})-f(X_t)|\Filt_{t}}\\
= &\, \EI{n,\nu}{\sint{t}{t+s}A^{n}f^{2}(X_{u})du | \Filt_{t}}\\
 &\, - 2f(X_t)\EI{n,\nu}{\sint{t}{t+s}A^{n}f(X_{u})du|\Filt_{t}}\\
\leq &\, s \| A^{n}f^{2} \|_{\infty} + 2s\|f\|_{\infty}\|A^{n}f\|_{\infty}\\
\leq &\, \delta \left( \|A^{n}f^{2}\|_{\infty} +
2\|f\|_{\infty}\|A^{n} f \|_{\infty} \right)\\
=: &\, \gamma_{n}(\delta),
\end{split}
\end{equation}
where we have used the fact that $(\mathbb{P}^{n,\nu})$ is a solution of
the martingale problem for $A^{n}$ and $f,f^{2} \in \Def(A^{n})$. Furthermore we notice
that
$(\Filt_{t})_{t\geq0}$ is finer as a filtration than
the canonical filtration of $(f(X_t))_{t\geq0}$. Therefore the inequality holds also
for the coarser filtration. We note that $\gamma_{n}(\delta)$ is simply a deterministic
random variable. To bound $\gamma_{n}(\delta)$ we use the fact that for $f \in
\testfcts$ for all $n \geq 1, x\in \Rd$ we have
\begin{equation}\label{eq:Anupperbound}
\begin{split}
 |A^{n}f(x)| = |\frac{1}{(2\pi)^{d/2}}\int
e^{i\scal{x}{\xi}}q^{n}(x,\xi)\hat{f}(\xi)d\xi| \\
\leq\frac{1}{(2\pi)^{d/2}} \int {|q^{n}(x,\xi)|}{|\hat{f}(\xi)|}d\xi \leq
\int C'
(1+ \|\xi
\|^{2}){|\hat{f}(\xi)|}d\xi.
\end{split}
\end{equation}
This follows from hypothesis \eqref{eq:tightnessconditions1}, where $C'>0$
can be chosen independently of $x$ and $n$. The integral is finite, because
$\hat{f}$ as a Fourier transform of an element of $\testfcts$ belongs to the
Schwartz
space $\mathcal{S}(\Rd)$ and the Schwartz space is closed under multiplication by a
polynomial.
Therefore we obtain
\begin{equation}
\begin{split}
 &\, \lims{\delta \gt 0} \sups{n \geq 1} \gamma_{n}(\delta) \\
 = &\, \lims{\delta \gt 0} \sups{n \geq 1} \delta ( \|A^{n}f^{2}\|_{\infty} +2
\|f\|_{\infty} \|A^{n}f\|_{\infty}) ) \\
\leq &\, \lims{\delta \gt 0} \sups{n \geq 1} \delta (\int C' (1+ \|\xi
\|^{2})\|){|\hat{f^{2}}(\xi)|}d\xi+2 \|f\|_{\infty} \int C' (1+ \|\xi
\|^{2})\|){|\hat{f}(\xi)|}d\xi) \\
= &\, 0.
\end{split}
\end{equation}
So conditions \eqref{RK2} and \eqref{RK3} follow and the distributions
$\sequence{(\mathbb{P}^{n,\nu}\circ f(X)^{-1})}{n\geq 1}$ are relatively compact for every $f \in
\testfcts$.
By proposition \ref{prop:relcomp2} the sequence of distributions
$\sequence{(\mathbb{P}^{n,\nu})}{n\geq1}$ is relatively compact too.
We note that as a by-product of our proof we have shown the uniform boundedness of
$\{A^{n}f\}_{n\geq1}$ for each $f\in \testfcts$.
\end{proof}
\subsection{An existence theorem}

The following theorem is our central result. It provides the existence of a solution of the martingale problem in sufficient generality to be applicable in different scenarios in the
following sections.
We denote by $\lambda$ the Lebesgue measure on $\Rd$.
\begin{thm}\label{thm:centralresult}
Let $\sequence{(\mathbb{P}^{n})}{n \geq 1}$ be a sequence of probability measures for which the canonical process is Feller
with corresponding generators $(A_{n},\Def(A_{n}))$, and identical
laws of $X_0$. Let
$\testfcts \subset \Def(A_{n})$ for all $n \in \mathbb{N}$,  and write $\{q^{n}\}_{n\geq1}$
for the sequence of their symbols.
The conditions
\eqref{eq:tightnesscondition0},
\eqref{eq:tightnessconditions1}, and
\eqref{eq:tightnessconditions2}  are assumed to hold for the symbols.
Furthermore assume
\begin{equation}
\label{eq:SymbolLowerBound}
 \lims{\|\xi\|\gtinf} \infs{n \geq 1} \frac{\infs{z \in \Reals} \Re
 q^{n}(z,\xi)}{\log(1+\|\xi\|)}= \infty. \tag{A4}
\end{equation}
Let $(A,\Def(A))$ be an operator on $\BoundedMeasurable$ with
$\Def(A) \subset \Def(A_{n})$ for all $n \geq 1$.
Assume that there exists a sequence $\sequence{U_{m}}{m \geq 1}$ of open sets, such that for
all $m\geq 1$
\begin{itemize}
 \item
\begin{equation}\label{eq:propertiesU}
  \lims{m \gtinf} \lambda(U_{m})=0, \quad \overline{U_{m+1}} \subset U_{m};
\tag{B1}
\end{equation}
\item
\begin{equation}\label{eq:restricteduniformcont}
 \lims{n \gtinf}\sups{x \in \overline{B_{m}(0)} \cap (U_{m})^{c}}
|(A^{n}-A)f(x)|=0 \quad
\text{for all} ~f \in \Def(A). \tag{B2}
\end{equation}
\end{itemize}
Then the $\SkoSpace$-martingale problem associated with $A$ has a solution for every
initial distribution.
\end{thm}

\begin{proof}
By proposition \ref{prop:weakconvergence} and conditions
\eqref{eq:tightnesscondition0},\eqref{eq:tightnessconditions1},
\eqref{eq:tightnessconditions2}
we see that the sequence of distributions $\sequence{\mathbb{P}^{n}}{n \geq 1}$ is
relatively compact. So without loss of generality we can assume that it is weakly convergent. For every $n \geq 1$ the
process $\mathbb{P}^{n}$ solves the
$\SkoSpace$-martingale problem associated with $A_{n}$ with respect to the canonical process. We write $\mathbb{P}$ for the
weak limit of the sequence, and intend to show that $\mathbb{P}$
is a solution of the $\SkoSpace$-martingale problem associated with $A$.
This amounts to showing for all $f \in \Def(A)$ that
 \begin{equation}
\left(f(X_t)-f(X_{0})-\sint{0}{t}Af(X_s)ds \right)_{t\geq0}
 \end{equation}
is a  $\mathbb{P}$-martingale with respect to the filtration generated by
$(X_t)_{t\geq0}$. This is equivalent to the condition that
\begin{equation}\label{eq:MG-Prob}
 0=\E{(f(X_{t_{j+1}})-f(X_{t_{j}})-\sint{t_{j}}{t_{j+1}}Af(X_s)ds)\sprod{k=1}
{j}h_{k}(X_{t_{k}})}
\end{equation}
for all $j \in \mathbb{N}$, $0\leq t_{1} \leq \ldots \leq t_{j} < t_{j+1}$ and $f
\in
\Def(A)$,$h_{1},\ldots,h_{j} \in \Cb$.
We assume in the sequel that $t_{i}, 1
\leq i \leq j+1,$ belong to $D(X) :=  \{ t \geq 0 ~|~ \Prob{X_t=X_{t_{-}}}=1 \}$. We infer from lemma 7.7 and theorem 7.8 from chapter 3 \cite{EthierKurtz} that $D(X)$ is countable
and the finite dimensional distributions of $(X^n_t)_{t\geq0, t\in D(X)}$ converge weakly
to those of $(X_t)_{t\geq0, t\in D(X)}$. So by the right continuity of
\begin{equation}
 \E{(f(X_{t_{j+1}})-f(X_{t_{j}})-\sint{t_{j}}{t_{j+1}}Af(X_s)ds)\sprod{k=1}
{j}h_{k}(X_{t_{k}})}
\end{equation}
considered as a function of a single family $t_{i}, 1
\leq i \leq j+1,$ we can assume again without loss of generality that the
$t_{i}, 1
\leq i \leq j+1,$ lie in $D(X)$. First we write
\begin{equation}
\begin{split}
 &\,
\E{\left(f(X_{t_{j+1}})-f(X_{t_{j}})-\sint{t_{j}}{t_{j+1}}Af(X_s
)ds\right)\sprod{ k=1 } { j
}h_{k}(X_{t_{k}})}\\
 = &\,
\E{(f(X_{t_{j+1}})-f(X_{t_{j}}))\sprod{k=1}{j}h_{k}(X_{t_{k}})}\\
- &\, \sint{t_{j}}{t_{j+1}} \E{Af(X_s)\sprod{k=1}{j}h_{k}(X_{t_{k}})}ds,
\end{split}
\end{equation}
where we have used Fubini's theorem and the linearity of the expectation.
We now study the different terms separately, starting with the first, which we can rewrite
as
\begin{equation}
\begin{split}
& \E{(f(X_{t_{j+1}})-f(X_{t_{j}}))\sprod{k=1}{j}h_{k}(X_{t_{k}})}\\
= & \lims{n \gtinf}
\En{(f(X_{t_{j+1}})-f(X_{t_{j}}))\sprod{k=1}{j}h_{k}(X_{t_{k}})}.
\end{split}
\end{equation}
This is true because of the weak convergence of the finite-dimensional distributions.\\
Now we claim that
\begin{equation}\label{eq:discontinuity_term}
 \E{Af(X_s)\sprod{k=1}{j}h_{k}(X_{t_{k}})}=\lims{n\gtinf}
\En{Af(X_s)\sprod{k=1}{j}h_{k}(X_{t_{k}})}
\end{equation}
for all $s \geq 0, s \in D(X)$. To show this we use the continuous mapping theorem
(cf. theorem 4.27,\cite{Kallenberg}).
So we first have to study the continuity properties of $Af$.
We have
\begin{equation}\label{eq:restricteduniformcont2}
 \lims{n \gtinf}\sups{x \in \overline{B_{m}(0)} \cap
(U_{m})^{c}}|A^{n}f(x)-Af(x)| = 0
\quad \text{for all} ~m \geq 1 \tag{B2}
\end{equation}
by assumption. For all $m \geq 1$ the space $(\bar{C}(\overline{B_{m}(0)} \cap
(U_{m})^{c}),{\|\cdot\|_{\infty}})$ is a Banach space, because $\overline{B_{m}(0)}
\cap (U_{m})^{c}$ is compact. The restriction of $Af$ to $\overline{B_{m}(0)} \cap
(U_{m})^{c}$ is continuous, because according to \eqref{eq:restricteduniformcont}
it can represented as a strong limit of elements from $(\bar{C}(\overline{B_{m}(0)} \cap
(U_{m})^{c}),{\|\cdot\|_{\infty}})$. Let $U:=\cap_{m\geq1}
U_{m}$. We want to show continuity of $Af$ in every $x \in U^{c}$.
For this it is enough to show that for every $x \in U^c$ there exists an $m \geq1$, such
that $x$ is contained in an open neighborhood that is a subset of $\overline{B_{m+1}(0)}
\cap (U_{m+1})^{c}$. For this we note that by definition of $U$ there
exists for every $x \in U^{c}$ an $m\geq1$, such that $x \in
\overline{B_{m}(0)}\cap
(U_{m})^{c}$. By assumption \eqref{eq:propertiesU}
\begin{equation}
 x \in \overline{B_{m}(0)}\cap
(U_{m})^{c} \subset  B_{m+1}(0)\cap
(\overline{U_{m+1}})^{c} \subset \overline{B_{m+1}(0)}\cap
(U_{m+1})^{c}
\end{equation}
holds. From this the continuity of $Af$ for every $x \in U^{c}$ follows.\\
It remains to show that
\begin{equation}
 \Prob{X_s \in U}=0
\end{equation}
for all $s > 0, s\in D(X)$.\\
For this we use the portmanteau lemma (theorem 3.25 \cite{Kallenberg}) and theorem
\ref{thm:TransitionDensities}. We denote by $p^{n}$ the transition density
related to $\mathbb{P}^{n}$, the existence of which is guaranteed by theorem
\ref{thm:TransitionDensities}. For $s >0,
s \in D(X)$ we have
\begin{equation}\label{eq:probsopen}
\begin{split}
 &\, \Prob{X_s \in U} \leq \liminfs{n\gtinf}\Probn{X_s \in U_{m}}\\
= &\, \liminfs{n\gtinf} \ints{U_{m}} p^{n}(t,X_{0},y')dy'\\
\leq &\, \liminfs{n\gtinf} \ints{U_{m}} \sups{x,y \in
\Reals^{d}}p^{n}(t,x,y)dy'\\
\leq &\,  \liminfs{n\gtinf} \lambda(U_{m}) \frac{1}{(4\pi)^{d}} \ints{\Reals}
\exp\left(
-\frac{t}{16}\infs{z\in \Reals} \Re q^{n}(z,\xi) \right) d\xi\\
\leq &\,  \lambda(U_{m}) \frac{1}{(4\pi)^{d}} \ints{\Reals}
\exp\left(
-\frac{t}{16}\infs{n \geq 1}\infs{z\in \Reals} \Re q^{n}(z,\xi) \right) d\xi.\\
\end{split}
\end{equation}
Using assumption \eqref{eq:SymbolLowerBound} we see that the integral in the last line is
bounded. Therefore the expression in the last line converges to $0$ for $m \gtinf$, and we have $\Prob{X_s
\in U}=0$ for all $s>0, s
\in D(X)$.\\
After establishing that at a fixed time $t \in D(X)$ the process $(X_t)_{t\geq0}$ is not
in a point where $Af$ might be discontinuous, we want to show that this is enough for
the convergence of our original term in \ref{eq:discontinuity_term}.
We define $g:\Reals^{j+1}\gt
\Reals$, $g(x):=Af(x_{j+1})\sprod{k=1}{j}h_{k}(x_{k})$. In view of
the continuous mapping theorem and the fact that $g$ is bounded and continuous on $U^c$,
we obtain for $s> 0, s \in D(X)$ that
\begin{equation}
 \lims{n\gtinf}\En{g(X_s)}=\E{g(X_s)},
\end{equation}
and so
\begin{equation}
 \E{Af(X_s)\sprod{k=1}{j}h_{k}(X_{t_{k}})}=\lims{n\gtinf}
\En{Af(X_s)\sprod{k=1}{j}h_{k}(X_{t_{k}})}.
\end{equation}
To bring the approximative generators into play we decompose the right side of the
preceding equation in the following way:
\begin{equation}
\begin{split}
&\, \lims{n\gtinf}
\En{Af(X_s)\sprod{k=1}{j}h_{k}(X_{t_{k}})} \\
= &\, \lims{n\gtinf} \left(
\En{A_{n}f(X_s)\sprod{k=1}{j}h_{k}(X_{t_{k}})}
+
\En{(A-A_{n})f(X_s)\sprod{k=1}{j}h_{k}(X_{t_{k}})}\right).
\end{split}
\end{equation}
We want to obtain further properties of the second term. In fact, for
$s>0, s \in D(X)$, $n \in \mathbb{N}$ and $m\geq1$ we have
\begin{align}
&\, |\En{(A-A_{n})f(X_s)\sprod{k=1}{j}h_{k}(X_{t_{k}})}| \notag \\
\leq &\, \En{|(A-A_{n})f(X_s)\sprod{k=1}{j}h_{k}(X_{t_{k}})|} \notag\\
\leq &\,  C \cdot \En{|(A-A_{n})f(X_s)|} \notag \displaybreak[0]\\
= &\, C \cdot \En{|(A-A_{n})f(X_s)|\Indicator_{(U_{m}
\cup (\overline{B_{m}(0)})^{c})}(X_s)} \notag\\
+ &\,  C \cdot \En{|(A-A_{n})f(X_s)|\Indicator_{(U_{m})^{c} \cap
\overline{B_{m}(0)}}(X_s)}\notag \displaybreak[0]\\
\leq &\, C \| (A-A_{n})f \|_{\infty} \cdot \Big( \Probn{X_s \in U_{m}}
\Big.\notag\\
+ &\left.\, \Probn{X_s \notin
\overline{B_{m}(0)}} \right)+ C \sups{x \in (U_{m})^{c} \cap
\overline{B_{m}(0)}}|(A-A_{n})f(x)|\notag \displaybreak[0]\\
\leq &\, C \sups{n \geq 1}\| (A-A_{n})f \|_{\infty} \cdot \left( \sups{n \geq
1}\Probn{X_s \in U_{m}}+ \right.\notag\\
&\, \left. \sups{n \geq 1}\Probn{X_s \notin
\overline{B_{m}(0)}} \right)+ C \sups{x \in (U_{m})^{c} \cap
\overline{B_{m}(0)}}|(A-A_{n})f(x)|.
\end{align}
In the preceding chain of inequalities we first argue by Jensen's inequality, then
bound the product of the bounded functions $h_{k}, 1\leq k\leq j,$  by the constant
$C>0$. Then we split the expectation into a part where $X_s$ lies in $U_{m}
\cup (\overline{B_{m}(0)})^{c}$ and a part where $X_s$ lies in $(U_{m})^{c}
\cap \overline{B_{m}(0)}$. For the last line assumption
\eqref{eq:restricteduniformcont2} gives
\begin{equation}
\lims{n \gtinf} C\sups{x \in
(U_{m})^{c} \cap
\overline{B_{m}(0)}}|(A-A_{n})f(x)|=0.
\end{equation}
Finally, we look at the next-to-last line. We have
\begin{equation}
 \sups{n \geq 1}\|(A-A_{n})f\|_{\infty} \leq \|Af\|_{\infty}+\sups{n \geq
1}\|A_{n}f\|_{\infty} < \infty,
\end{equation}
since $Af$ and also $ \sups{n \geq 1} \|A_{n}f\|_{\infty}$
is bounded, as we saw in the proof of \ref{prop:weakconvergence}.
In conclusion, we see that
\begin{equation}
\begin{split}
 &\, \lims{n
\gtinf}|\En{(A-A_{n})f(X_s)\sprod{k=1}{j}h_{k}(X_{t_{k}})}|\\
\leq &\, C \sups{n \geq 1}\| (A-A_{n})f \|_{\infty} \cdot ( \sups{n \geq
1}\Probn{X_s \in U_{m}}+
\sups{n \geq 1}\Probn{X_s \notin
\overline{B_{m}(0)}} )+ \\
&\, C \lims{n \gtinf}\sups{x \in (U_{m})^{c} \cap
\overline{B_{m}(0)}}|(A-A_{n})f(x)|\\
= &\,  C \sups{n \geq 1}\| (A-A_{n})f \|_{\infty} \cdot ( \sups{n \geq
1}\Probn{X_s \in U_{m}}+
\sups{n \geq 1}\Probn{X_s \notin
\overline{B_{m}(0)}} ).
\end{split}
\end{equation}
For $m \gtinf$ we have $\sups{n \geq
1}\Probn{X_s \in U_{m}} \gt 0$ and $\sups{n \geq 1}\Probn{X_s \notin
\overline{B_{m}(0)}} \gt 0$ for almost all $s\geq0$ (up to a set of Lebesgue measure
zero).
This follows from an argument similar to the one in \eqref{eq:probsopen} and
condition \eqref{RK1}, which holds in the underlying case.
So we see that $\lims{n
\gtinf}|\En{(A-A_{n})f(X_s)\sprod{k=1}{j}h_{k}(X_{t_{k}})}|=0$.
Therefore we obtain
\begin{equation}
\lims{n\gtinf}
\En{Af(X_s)\sprod{k=1}{j}h_{k}(X_{t_{k}})}
= \lims{n\gtinf}
\En{A_{n}f(X_s)\sprod{k=1}{j}h_{k}(X_{t_{k}})}.
\end{equation}
\\
Now we summarize our results to show equation
\eqref{eq:MG-Prob}. Indeed,
\begin{align}
&\, \E{(f(X_{t_{j+1}})-f(X_{t_{j}})-\sint{t_{j}}{t_{j+1}}Af(X_s)ds)\sprod{k=1}
{j}h_{k}(X_{t_{k}})} \notag\\
= &\,\E{(f(X_{t_{j+1}})-f(X_{t_{j}}))\sprod{k=1}{j}h_{k}(X_{t_{k}})}
-\sint{t_{j}}{t_{j+1}}\E{Af(X_s)\sprod{k=1}{j}h_{k}(X_{t_{k}})ds}\displaybreak[0]
\notag\\
= &\, \lims{n\gtinf}
\En{(f(X_{t_{j+1}})-f(X_{t_{j}}))\sprod{k=1}{j}h_{k}(X_{t_{k}})} \displaybreak[0] \notag\\
- &\,
\sint{t_{j}}{t_{j+1}}\lims{n\gtinf}
\En{Af(X_s)\sprod{k=1}{j}h_{k}(X_{t_{k}})}ds\displaybreak[0] \notag\\
= &\, \lims{n\gtinf}
\En{(f(X_{t_{j+1}})-f(X_{t_{j}}))\sprod{k=1}{j}h_{k}(X_{t_{k}})}
\notag \\
&\, - \sint{t_{j}}{t_{j+1}}\lims{n\gtinf}\left(
\En{A_{n}f(X_s)\sprod{k=1}{j}h_{k}(X_{t_{k}})} \right.\notag\\
+ &\, \left. \En{(A-A_{n})f(X_s)\sprod{k=1}{j}h_{k}(X_{t_{k}})}\right)ds\notag
\displaybreak[0]\\
= &\,   \lims{n \gtinf}
\En{(f(X_{t_{j+1}})-f(X_{t_{j}}))\sprod{k=1}{j}h_{k}(X_{t_{k}})}
\notag\\
&\, - \sint{t_{j}}{t_{j+1}}\lims{n\gtinf}
\En{A_nf(X_s)\sprod{k=1}{j}h_{k}(X_{t_{k}})}ds\displaybreak[0] \notag\\
= &\,   \lims{n \gtinf}
\En{f(X_{t_{j+1}})-f(X_{t_{j}}))\sprod{k=1}{j}h_{k}(X_{t_{k}})}
\notag\\
&\, - \lims{n\gtinf}\sint{t_{j}}{t_{j+1}}
\En{A_nf(X_s)\sprod{k=1}{j}h_{k}(X_{t_{k}})}ds\displaybreak[0] \notag\\
= &\,
\lims{n\gtinf}\En{\left(f(X_{t_{j+1}})-f(X_{t_{j}})-
\sint{t_{j}}{t_{j+1}}
A_{n}f(X_s)ds\right)\sprod{k=1}{j}h_{k}(X_{t_{k}})}\displaybreak[0] \notag\\
= &\, 0.
\end{align}
Interchanging integrals, expectations and limits are justified by Fubini's theorem and
dominated convergence. So we can conclude that
equation \eqref{eq:MG-Prob} holds for all $j \in \mathbb{N}$, $0\leq t_{1} \leq \ldots
\leq t_{j} < t_{j+1}$ and $f
\in
\Def(A)$,$h_{1},\ldots,h_{j} \in \Cb$, and so $\mathbb{P}$ solves the
martingale problem associated with $A$.
\end{proof}

\begin{remark}
The solutions obtained in theorem \ref{thm:centralresult} can be shown to be quasi-left
continuous (theorem 3.12, chapter 4, \cite{EthierKurtz}), i. e. for every sequence of
bounded predictable non-decreasing stopping times $\sequence{\tau_k}{k\geq1}$ with
$\lims{k \gtinf} \tau_k=\tau$ we have $\Prob{\lims{k\gtinf}X_{\tau_k}=X_{\tau}}=1$. This
implies that $(X_{t})_{t\geq0}$ has a.s. no fixed times of discontinuity under $\mathbb{P}$, and the set
$D(X)$ of our
proof is in fact empty.
\end{remark}

\section{Applications of the existence theorem}\label{sec:anwendung}
All our applications are stated for a one-dimensional setting. Generalizations to
multidimensional scenarios should be possible. We first discuss the glueing of two L\'evy processes at a threshold where their dynamics meet in a discontinuous way.
\subsection{Glueing together of two L\'evy processes}

\subsubsection{Solutions of the martingale problem}
We consider two independent L\'evy processes
$(L^{1}_{t})_{t\geq0}$,$(L^{2}_{t})_{t\geq0}$
with symbols $q^{(1)}$,$q^{(2)}$ and generators $A^{(1)}$,$A^{(2)}$. We assume that both
symbols satisfy the Hartman-Wintner condition:
\begin{equation}\label{eq:hartmannwintner2}
  \lims{|\xi| \gtinf} \frac{\Re q^{i}(\xi)}{\log(1+|\xi|)}= \infty, i \in
\{1,2\}.
\end{equation}
Now we try to construct a stochastic process which behaves like
$(L^{1}_{t})_{t\geq0}$ below $0$ and like
$(L^{2}_{t})_{t\geq0}$ above $0$. We define $I_{1}:=(-\infty,0]$,
$I_{2}:=(0,\infty)$ and
\begin{equation}
 Af(x):=\IndOne(x)A^{(1)}f(x)+\IndTwo(x)A^{(2)}f(x), x \in \Reals
\end{equation}
with $f \in \Def(A):= C^{\infty}_{c}(\Reals)$.

\begin{thm}\label{thm:martingaleprobglueing}
The $D_{\Reals}([0,\infty))$-martingale problem associated with $A$ has a solution for
every initial distribution.
\end{thm}
\begin{proof}
We want to apply our general existence theorem. To this end we first have to choose an
approximating sequence of Feller processes. We define
$\psi:\Reals^{2} \gt \Reals$ by
\begin{equation}
 \psi (\xi) :=
q^{(1)}(\xi_{1})+q^{(2)}(\xi_{2}), \xi = (\xi_1, \xi_2)\in\mathbb{R}^2
\end{equation}
and $\phi_{n}: \Reals \gt \Reals^{2}$,
\begin{equation}
 \phi_{n}^{T}(x) := (g^{n}_{1}(x),
g^{n}_{2}(x)), x \in \Reals,
\end{equation}
where
\begin{equation}
g^{n}_{1}(x)=
\begin{cases}
  0 \quad ,x\geq1/n\\
  1-nx \quad, 0 < x < 1/n\\
  1 \quad, x \leq 0
\end{cases}
\end{equation}
and
$g^{n}_{2}=1-g^{n}_{1}$. For $n\in\mathbb{N}$ one observes that $\phi_{n}$ is Lipschitz-continuous and bounded, and $\psi$ is the symbol of the L\'evy process
$\bigl((L^{1}_{t},L^{2}_{t})\bigr)_{t\geq0}$, and therefore a continuous negative definite
function with $c=0$. So the two conditions for proposition \ref{thm:glueing} are fulfilled
for all $n\geq1$. Therefore there exists a sequence of probability measures $(\mathbb{P}^{n})$ for which the canonical process is Feller
with generators
$\{A_n\}_{n \geq1}$, for which $C_{c}^{\infty}(\Reals) \subset
\Def(A_{n}), n\in\mathbb{N},$ and whose associated symbols are of the form
\begin{equation}\label{eq:symbolglueing}
 q^{n}(x,\xi)=q^{(1)}(g^{n}_{1}(x)\xi)+q^{(2)}(g^{n}_{2}(x)\xi),x, \xi \in
\Reals, n \geq 1.
\end{equation}
We want to show that this sequence satisfies the conditions of
theorem \ref{thm:centralresult}. One can see that condition
\eqref{eq:tightnessconditions1} holds, because
there exist $C^{1},C^{2},C>0$, such that  for $\xi \in \Reals$
\begin{equation}
\begin{split}
 q^{n}(x,\xi)=&\, q^{(1)}(g^{n}_{1}(x)\xi)+q^{(2)}(g^{n}_{2}(x)\xi)\\
 \leq &\, C^{1}(1+ |g^{n}_{1}(x)\xi|^{2}) + C^{2}(1+ |g^{n}_{2}(x)\xi|^{2}) \leq
C
(1+|\xi|^{2}).
\end{split}
\end{equation}
Furthermore condition \eqref{eq:tightnesscondition0} holds by the properties of symbols of
L\'evy processes, and also condition \eqref{eq:tightnessconditions2} holds in view of the
continuity of $q^{(1)}$ and $q^{(2)}$. To show condition \eqref{eq:SymbolLowerBound},
we observe that
\begin{equation}
\begin{split}
  & \lims{|\xi |\gtinf} \infs{n \geq1}\frac{ \infs{x \in \Reals}\Re
q^{n}(x,\xi)}{\log(1+|\xi|)}\\
 = & \lims{|\xi |\gtinf} \infs{n \geq1}\infs{x \in \Reals}\frac{\Re
q^{(1)}(g^{n}_{1}(x)\xi)+\Re q^{(2)}(g^{n}_{2}(x)\xi)}{\log(1+|\xi|)}\\
\geq  & \lims{|\xi |\gtinf} \mins{i \in \{1,2\}} \infs{\frac{1}{2} \leq a \leq
1} \frac{\Re q^{(i)}(a\xi)}{\log(1+|\xi|)}\\
= & \lims{|\xi' |\gtinf} \mins{i \in \{1,2\}}\infs{\frac{1}{2} \leq a \leq
1}\frac{\Re
q^{(i)}(\xi')}{\log(1+\frac{1}{a}|\xi'|)}\\
\geq & \lims{|\xi' |\gtinf} \mins{i \in \{1,2\}} \infs{\frac{1}{2} \leq a \leq
1}\frac{\Re
q^{(i)}(\xi')}{\log(\frac{1}{a}(1+|\xi'|))}\\
\geq & \lims{|\xi' |\gtinf} \mins{i \in \{1,2\}}  \frac{\Re
q^{(i)}(\xi')}{\log(2)+\log(1+|\xi'|)}\\
= & ~ \infty,
\end{split}
\end{equation}
because the Hartman-Wintner condition is satisfied by assumption for the two L\'evy
processes and $1 \geq \max \{g_1^n(x),g_2^n(x)\}\geq\frac{1}{2}$ for all $x
\in
\Reals,n \in \mathbb{N}$.
This shows that the condition \eqref{eq:SymbolLowerBound} is fulfilled.
Looking at the definition of the operator $(A,\Def(A))$ above, we see that the
conditions on $\Def(A)$ are also satisfied. One observes that for an $f \in
C^{\infty}_{c}(\Reals)$ the relation $Af \in
B(\Reals)$ holds, because $A^{(1)}f \in \bar{C}(\Reals)$ and $A^{(2)}f \in
\bar{C}(\Reals)$.\\
Now it remains to show that there exists a sequence of open sets
$\{U_{m}\}_{m\geq1}$,
such that $\lims{m\gtinf}\lambda(U_{m})=0$ and $\overline{U_{m+1}}\subset U_{m}$
for all $m \geq 1$, and
\begin{equation}
  \lims{n \gtinf}\sups{x \in \overline{B_{m}(0)} \cap (U_{m})^{c}}
|(A^{n}-A)f(x)|=0, \quad
\text{for all} ~f \in \Def(A), m\geq 1.
\end{equation}
To this end for every $m \in \mathbb{N}$ we define $U_{m}:=(-1/m,1/m)$.
Obviously
$\lims{m\gtinf}\lambda(U_{m})=0$ and $\overline{U_{m+1}}\subset U_{m}$.
The form of the symbols implies that for $f \in
C^{\infty}_{c}(\Reals)$ and $n \geq m$ the equations
$A^{n}f|_{(-\infty,-1/m]}=A^{(1)}f|_{(-\infty,-1/m]}=Af|_{(-\infty,-1/m]}$ and
$A^{n}f|_{[1/m,\infty)}=A^{(2)}f|_{[1/m,\infty)}=Af|_{[1/m,\infty)}$ hold. Therefore we
have $|(A^{n}-A)f(x)|=0$ for all $ x \in (U_{n})^{c}$
and $f \in \Def(A)$, and so for all $m\geq 1$
\begin{equation}
  \lims{n \gtinf}\sups{x \in \overline{B_{m}(0)} \cap (U_{m})^{c}}
|(A^{n}-A)f(x)|=0.
\end{equation}
In summary, we have verified all the conditions enabling us to apply theorem
\ref{thm:centralresult} to $(A,\Def(A))$. Therefore the
$D_{\Reals}([0,\infty))$-martingale problem associated with $(A,\Def(A))$ has a solution.

\end{proof}

\subsubsection{Existence of solutions for associated SDE}
We now show that a solution of the martingale problem provides a weak solution of the associated SDE. For this purpose, we continue to consider two independent L\'evy processes
$(L^{1}_{t})_{t\geq0}$,$(L^{2}_{t})_{t\geq0}$. We want to show that under the conditions discussed a
solution of the martingale problem from theorem \ref{thm:martingaleprobglueing} is also a
weak solution of an associated SDE. We consider the following SDE
\begin{equation}\label{eq:LevySDE0}
X_t=X_{0}+\sint{0}{t}
\IndOne(X_{s^{-}})dL_{s}^{1}+\sint{0}{t}\IndTwo(X_{s^{-}})dL_{s}^{2}, t\geq0.
\end{equation}

A weak solution of this SDE will consist of two independent L\'evy processes
$(L^{*,1}_{t})_{t\geq0}$,$(L^{*,2}_{t})_{t\geq0}$ and a process $(X_t)_{t\geq0}$, such
that the processes are defined and adapted on a filtered probability space
$(\Omega,\Filt,(\Filt_{t})_{t\geq0},\mathbb{P})$ and such that $(L^{*,1}_{t})_{t\geq0}
\equaldist
(L^{1}_{t})_{t\geq0}$,$(L^{*,2}_{t})_{t\geq0} \equaldist (L^{2}_{t})_{t\geq0}$
and for all $t \geq0$
\begin{equation}\label{eq:LevySDE2}
X_t=X_{0}+\sint{0}{t}
\IndOne(X_{s^{-}})dL_{s}^{*,1}+\sint{0}{t}\IndTwo(X_{s^{-}})dL_{s}^{*,2}
\end{equation}
holds $\mathbb{P}$-almost surely.

\begin{thm}\label{thm:weakexistencelevy}
Let $(L^{1}_{t})_{t\geq0}$ and $(L^{2}_{t})_{t\geq0}$ be two independent L\'evy processes,
which satisfy condition \eqref{eq:hartmannwintner2}. Then the stochastic differential
equation
\begin{equation}\label{eq:LevySDE1}
X_t=X_{0}+\sint{0}{t}
\IndOne(X_{s^{-}})dL_{s}^{1}+\sint{0}{t}\IndTwo(X_{s^{-}})dL_{s}^{2}, t\geq0,
\end{equation}
possesses a weak solution for every initial distribution.
\end{thm}
 \begin{proof}
We consider the operator $A$ from the preceding section and aim at showing that the solution of the
martingale problem associated with $A$ is also (part of) a weak solution for the
considered SDE.
We may represent $A$ on $f \in \Def(A)$ by
\begin{equation}\label{eq:OperatorEqOriginal2}
\begin{split}
 Af(x)=\frac{1}{2}a(x)\partial_{x}^{2}f(x)+b(x)\partial_{x}f(x)\\
+\ints{\Reals}
(f(x+y)-f(x)-y\partial_{x}f(x)\Indicator_{\overline{B_{1}(0)}}(y))\eta(x,dy), x
\in \Reals,
\end{split}
\end{equation}
with
\begin{equation}
 \eta(x,dy)=\IndOne(x)\nu^{1}(dy)+\IndTwo(x)\nu^{2}(dy),
\end{equation}
$a(x)=\IndOne(x)a^{1}+\IndTwo(x)a^{2}$,
and $b(x)=\IndOne(x)b^{1}+\IndTwo(x)b^{2}$.
Now we define the metric space $S:=\Reals\times\{1,2\}\subset \Reals^2$
with metric induced by $(\Reals^2,\|\cdot\|)$. We observe that every $B \in
\mathcal{B}(S)$ can be written as $B=B_1 \times \{1\} \cup B_2\times \{2\}, B_1,B_2 \in
\mathcal{B}(\Reals)$ in a unique way. The first coordinate in $S$ represents the jump
height, while the second describes the association of the jump with one of the
two L\'evy processes. Now we define $\nu$ on $S$ as $\nu(B):=\nu^{1}(B_1) +\nu^{2}(B_2), B
\in \mathcal{B}(S)$. This measure is $\sigma$-finite, because both $\nu_1$ and $\nu_2$ are
$\sigma$-finite.\\
Further we define
\begin{equation}
 \zeta(x,u):=\IndOne(x)\Indicator_{\Reals \times
\{1\}}(u)+\IndTwo(x)\Indicator_{\Reals \times \{2\}}(u),x \in \Reals, u \in S,
\end{equation}
and $\gamma(u):=u_{1}, u=(u_1,u_2)\in S$.
One can see that
\begin{equation}
 \eta(x,\Gamma)=\ints{S} \zeta(x,u) \Indicator_{\Gamma}(\gamma(u)) \nu(du),
\Gamma \in
\mathcal{B}(\Reals).
\end{equation}
We decompose $S=S_1 \cup S_2$ with $S_1:=[-1,1] \times \{1,2\}$ and
$S_2:=[-1,1]^{c} \times \{1,2\}$. One notices directly that
\begin{equation}
 \Indicator_{S_1}(u)=\Indicator_{[-1,1]}(\gamma(u))
\end{equation}
holds and
\begin{equation}
 \ints{S_1} \zeta(x,u)\gamma(u)^{2}du+\ints{S_2} \zeta(x,u)\nu(du)
<\infty
\end{equation}
holds for all $x\in \Reals$. This is due to the fact that $\nu^1$ and $\nu^2$
are L\'evy measures.\\
Further $\mathcal{R}(A)\subset B(\Reals)$ as we observed in the preceding section and for
all compact sets $K\subset \Reals$ we have
\begin{equation}
\begin{split}
  &\, \sups{x \in K}\left(
a(x)+|b(x)|+\ints{S_1}\zeta(x,u)\gamma(u)^2\nu(du)+\ints{S_2}\zeta(x,
u){|\gamma(u)|\wedge 1}\nu(du)\right)\\
\leq &\, a^{1}+a^{2}+|b^{1}|+|b^{2}|\\
+ &\, \ints{\Reals}(1 \wedge
|y|^{2})\nu^{1}(dy)+\ints{\Reals}(1 \wedge
|y|^{2})\nu^{2}(dy)< \infty.\\
\end{split}
\end{equation}
This follows from the properties of the L\'evy measures $\nu^{1},\nu^{2}$.
Therefore all conditions for the application of theorem \ref{thm:KurtzEquivalence}
are satisfied. There exists a filtered probability space $(\Omega,\Filt,
(\Filt_{t})_{t\geq0},\mathbb{P})$ on which a solution of the martingale problem
$(X_t)_{t\geq0}$, a two-dimensional standard Brownian motion
$(B_{t})_{t\geq0}=(B^{1}_{t},B^{2}_{t})_{t\geq0}$ and an independent Poisson point
process $(p_{t})_{t\geq0}$ on $S$ with intensity measure $\nu$ and associated Poisson
random measure $\xi$ can be defined, such that all three processes are
$(\Filt_t)$-adapted and the equation
\begin{equation}
\begin{split}
 X_t= X_{0}+ \sint{0}{t}\sigma_{1}(X_s)dB^{1}_{s} +
\sint{0}{t}\sigma_{2}(X_s)dB^{2}_{s} +
\sint{0}{t}b(X_s)ds\\
 + \sint{0}{t}\ints{S_1}
(\IndOne(X_{s^{-}})\Indicator_{\Reals
\times \{1\}}(u)+\IndTwo(X_{s^{-}})\Indicator_{\Reals
\times \{2\}}(u))\gamma(u)\tilde
{\xi}(ds,du)\\
 + \sint{0}{t}\ints{S_2}
(\IndOne(X_{s^{-}})\Indicator_{\Reals
\times \{1\}}(u)+\IndTwo(X_{s^{-}})\Indicator_{\Reals
\times \{2\}}(u))\gamma(u)\xi(ds,du),
\end{split}
\end{equation}
holds $\mathbb{P}$-almost surely with $\sigma_{1}(x):=\sqrt{a^{1}}\IndOne(x),x \in
\Reals$ and
$\sigma_{2}(x):=\sqrt{a^{2}}\IndTwo(x),x \in \Reals$.\\
Now we want to show that the process $(X_t)_{t\geq0}$ is also a weak solution of the SDE
\eqref{eq:LevySDE1}. To show this we define $(L^{*,1}_{t})_{t\geq0}$ and
$(L^{*,2}_{t})_{t\geq0}$ with the help of the L\'evy-It\^o decomposition in terms of $\xi$
and
$(B_{t})_{t\geq0}$. It makes sense to use the individual components of the Brownian
motion $(B_{t})_{t\geq0}$ to construct the Brownian motion parts of the L\'evy
processes. The question of how to construct the (compensated) jump parts of the
processes from $\xi$ remains.
We define
\begin{equation}
 \xi^{1}(B):=\xi(B \times \{1\}), \quad
\xi^{2}(B):=\xi(B
\times \{2\}), B \in \mathcal{B}([0,\infty)) \otimes \mathcal{B}(\Reals).
\end{equation}
One observes that the random measures thus defined are independent Poisson random measures with intensity measures $\lambda \otimes
\nu^{1},\lambda \otimes \nu^2$.\\
So we define
\begin{equation}
 L^{*,i}_{t}:=b^i t+\sqrt{a^i}B^{i}_{t}+\sint{0}{t}\ints{[-1,1]^{c}}
y\xi^{i}(ds,dy)+\sint{0}{t}\ints{[-1,1]} y\tilde{\xi}^{i}(ds,dy),t \geq 0, i\in
\{1,2\}.
\end{equation}
These are independent L\'evy processes with the triples
$(a^i,b^i,\nu^{i}), i \in \{1,2\}$.
So we see for $t \geq 0$
\begin{align}
 & X_{0}+\sint{0}{t}
\IndOne(X_{s^{-}})dL_{s}^{*,1}+\sint{0}{t}\IndTwo(X_{s^{-}})dL_{s}^{*,2}\notag\\
= &  X_{0}+ \sint{0}{t}\sqrt{a_{1}}\IndOne(X_{s^{-}})dB^{1}_{s} +
\sint{0}{t}\sqrt{a_{2}}\IndTwo(X_{s^{-}})dB^{2}_{s}\notag\\
+ &\sint{0}{t}b_{1}\IndOne(X_{s^{-}})ds+\sint{0}{t}b_{2}\IndTwo(X_{s^{-}}
)ds\notag\\
+ &  \sint{0}{t}\ints{[-1,1]}\IndOne(X_{s^{-}})y\tilde{\xi}^{1}(ds,dy)
 + \sint{0}{t}\ints{[-1,1]}
\IndTwo(X_{s^{-}})y\tilde{\xi}^{2}(ds,dy)\notag\\
+ &  \sint{0}{t}\ints{[-1,1]^{c}}\IndOne(X_{s^{-}})y
\xi^{1}(ds,dy)
 + \sint{0}{t}\ints{[-1,1]^{c}}
\IndTwo(X_{s^{-}})y\xi^{2}(ds,dy)\displaybreak[0] \notag\\
= &  X_{0}+ \sint{0}{t}\sqrt{a_{1}}\IndOne(X_{s^{-}})dB^{1}_{s} +
\sint{0}{t}\sqrt{a_{2}}\IndTwo(X_{s^{-}})dB^{2}_{s}\notag\\
+ & \sint{0}{t}b_{1}\IndOne(X_{s^{-}})ds+\sint{0}{t}b_{2}\IndTwo(X_{s^{-}}
)ds\notag\\
+ & \sint{0}{t}\ints{S_1}
\IndOne(X_{s^{-}})\Indicator_{\Reals
\times \{1\}}(u)\gamma(u)\tilde{\xi}(ds,du)
 +
\sint{0}{t}\ints{S_1}\IndTwo(X_{s^{-}})\Indicator_{\Reals \times
\{2\}}(u)\gamma(u)\tilde{\xi}(ds,du)\notag\\
+ & \sint{0}{t}
\ints{S_2}\IndOne(X_{s^{-}})\Indicator_{\Reals
\times \{1\}}(u)\gamma(u)\xi(ds,du)
 + \sint{0}{t}\ints{S_2}\IndTwo(X_{s^{-}})\Indicator_{\Reals \times
\{2\}}(u)\gamma(u)\xi(ds,du)\displaybreak[0]\\
= & X_{0}+ \sint{0}{t}\sigma_{1}(X_s)dB^{1}_{s} +
\sint{0}{t}\sigma_{2}(X_s)dB^{2}_{s} +
\sint{0}{t}b(X_s)ds\notag\\
 + &  \sint{0}{t}\ints{S_1}
(\IndOne(X_{s^{-}})\Indicator_{\Reals
\times \{1\}}(u)+\IndTwo(X_{s^{-}})\Indicator_{\Reals
\times \{2\}}(u))\gamma(u)\tilde
{\xi}(ds,du)\notag\\
 + & \sint{0}{t}\ints{S_2}
(\IndOne(X_{s^{-}})\Indicator_{\Reals
\times \{1\}}(u)+\IndTwo(X_{s^{-}})\Indicator_{\Reals
\times \{2\}}(u))\gamma(u)\xi(ds,du)\notag\\
=  &  X_t, \mathbb{P}\text{-almost surely}. \notag
\end{align}
This shows that \ref{eq:LevySDE1} has a weak solution.
\end{proof}
\begin{corollary}
 Let $(S^{1}_t)_{t\geq0}$ and $(S^{2}_t)_{t\geq0}$ be two independent
symmetric $\alpha$-stable processes. The stochastic differential equation
\begin{equation}
  X_t=X_{0}+\sint{0}{t}
\IndOne(X_{s^{-}})dS_{s}^{1}+\sint{0}{t}\IndTwo(X_{s^{-}})dS_{s}^{2}
\end{equation}
has a weak solution.
\end{corollary}
\begin{proof}
Both processes fulfill the condition \eqref{eq:hartmannwintner2}.
\end{proof}

\subsection{Stable-like processes}\label{sec:anwendungstable}
In this subsection we apply our general existence theorem for solutions of martingale problems to stable-like processes with discontinuous stability index function. To this end we turn to the operator introduced in \ref{sec:stablelike}: \begin{equation}
 Af(x):=\ints{\Reals\backslash \{0\}}
(f(x+y)-f(x)-y\partial_xf(x)\Indicator_{[-1,1]\backslash\{0\}}(y))\frac{h(x)}{|y|^{
1+\alpha(x)}}dy, x \in \Reals,
\end{equation}
with $f \in \Def(A)=C_{c}^{\infty}(\Reals)$. We recall that $h:\Reals
\gt \Reals$ is defined via the equation
\begin{equation}
 1=\ints{\Reals\backslash\{ 0\}}(1-\cos(y))\frac{h(x)}{|y|^{1+\alpha(x)}}dy, x
\in \Reals.
\end{equation}
$A$ is a pseudo-differential operator, whose symbol can be written as
\begin{equation}
 q(x,\xi)=|\xi|^{\alpha(x)} , x,\xi \in \Reals.
\end{equation}
We will formulate an existence result for the $\SkoSpace$-martingale problem associated
with $A$ in the case of a not-necessarily continuous parameter function $\alpha$.
\begin{thm}\label{thm:stablelikeapplication}
Let $\alpha:\Reals \gt \Reals$ be a measurable function and let $D$ denote its set of
discontinuities. If its closure $\overline{D}$ is countable and
\begin{equation}\label{eq:StabilityBoundedness}
0 < \infs{x \in \Reals} \alpha(x) \leq \sups{x \in \Reals} \alpha(x)<
2,
\end{equation}
then the
$D_{\Reals}([0,\infty))$-martingale problem associated to the operator
\begin{equation}
 Af(x):=\ints{\Reals\backslash \{0\}}
(f(x+y)-f(x)-y\partial_xf(x)\Indicator_{[-1,1]\backslash\{0\}}(y))\frac{h(x)}{|y|^{
1+\alpha(x)}}dy, x \in \Reals,
\end{equation}
with $\Def(A)= C_{c}^{\infty}(\Reals)$ has a solution for every initial distribution.
\end{thm}
\begin{proof}
We will again show that there exists an approximating sequence of Feller processes so
that we can apply theorem \ref{thm:centralresult}.\\
First we show that for every function
$\alpha:\Reals \gt \Reals$ which satisfies the assumptions, there exists a sequence of
functions
$\sequence{\alpha^{n}}{n \geq 1}$ in $\overline{C}^{1}(\Reals)$ and a sequence of open
sets $\sequence{U_{m}}{m\geq 1}$, such that
\begin{itemize}
 \item \begin{equation}\label{eq:approxsequence0}
        \overline{U_{m+1}} \subset U_{m} \quad \text{for
all}~m\geq 1 ~,~
\lims{m \gtinf}\lambda(U_{m})=0, \tag{S1}
       \end{equation}

 \item \begin{equation}\label{eq:approxsequence1}
        0 < \infs{n \geq 1}\infs{x \in \Reals} \alpha^{n}(x) \leq \sups{n\geq 1}
\sups{x \in \Reals} \alpha^n(x)< 2, \tag{S2}
       \end{equation}
\item \begin{equation}\label{eq:approxsequence2}
        \lims{n \gtinf}\sups{x \in [-m,m] \cap (U_{m})^{c}}
|\alpha^{n}(x)-\alpha(x)|=0 \quad \text{for all}~ m \geq 1. \tag{S3}
      \end{equation}
\end{itemize}
\subsubsection*{Construction of the sets $U_{m}$}
We will use concepts from appendix \ref{sec:topologicalconcepts} in the setting of
the metric space $(\Reals,d)$ with $d(x,y):=|x - y|,x,y \in \Reals$.
We use basic knowledge about \emph{derived sets} in the topological sense to show that
$\overline{D}^{(\infty)}$ is empty.
\begin{lemma}
 $\overline{D}^{\infty}$  is empty.
\end{lemma}
\begin{proof}
Assuming the opposite, we conclude by theorem \ref{thm:CantorResultat} and the fact
that $\overline{D}^{\infty}$ is closed as an intersection of closed sets, that
$\overline{D}^{\infty}$ has to be uncountable. This is a contradiction because
$\overline{D}^{\infty} \subset \overline{D}$ and by assumption $\overline{D}$ is
countable.
\end{proof}
$\overline{D}^{\infty}$ being empty means by the hierarchy of the
$\overline{D}^{(i)}, i\in\mathbb{N}$, that there exists $k \geq 0$ such that
$\overline{D}^{(l)}=\varnothing$ for all $l \geq k$. So we deduce
\begin{equation}
 \overline{D}=\sbigcup{i=0}{k-1}\left(\overline{D}^{(i)}\setminus\overline{D}^{
(i+1) }\right)
.
\end{equation}
The construction of the sets implies that each $\overline{D}^{(i)}
\setminus \overline{D}^{(i+1)}$ consists of isolated points.
Because $\overline{D}$ is countable, we can write it as
\begin{equation}
 \overline{D}= \sbigcup{i=0}{k-1}\bigcups{1 \leq j \leq k_{i}}\{d^{i}_{j}\}
\end{equation}
with appropriately chosen $d^{i}_{j} \in \overline{D}^{(i)}
\setminus \overline{D}^{(i+1)}$ and $k_i \in \mathbb{N} \cup \{\infty\}$.
The next step is to approximate this set with an open set. For $m\in\mathbb{N}$ define
\begin{equation}
U_{m}:= \sbigcup{i=0}{k-1}\bigcups{1 \leq j \leq
k_{i}}(d^{i}_{j}-r^{m,i}_{j},d^{i}_{j}+r^{m,i}_{j}),
\end{equation}
with $r^{m,i}_{j}= \frac{1}{m2^{j+2}} \left(1 \wedge
d(d^{i}_{j},(\overline{D}^{(i)}\backslash\overline{D}^{(i+1)})\backslash\{d^{i}
_ { j }
\})\right)$ for $m\in\mathbb{N}, 0\le i\leq k-1, 1\leq  j\leq k_i$.

The sets $\overline{D}^{(i)}\setminus\overline{D}^{(i+1)}$ consist of isolated points, so
the sets $(d^{i}_{j}-r^{m,i}_{j},d^{i}_{j}+r^{m,i}_{j})$ are open for all $m\in\mathbb{N},
0\le i\leq k-1, 1\leq  j\leq k_i$. This means that $U_{m}$ is open too.
We use distributivity of the closure for finite unions of sets to state
\begin{equation}
 \overline{U_{m}}=\sbigcup{i=0}{k-1}\overline{\bigcups{1 \leq j \leq
k_{i}}(d^{i}_{j}-r^{m,i}_{j},d^{i}_{j}+r^{m,i}_{j})}.
\end{equation}
The following lemma will be applied to the sets
$\bigcups{1 \leq j \leq k_i }(d^{i}_{j}-r^{m,i}_{j},d^{i}_{j}+r^{m,i}_{j})$ of
$\overline{D}^{(i)}\setminus\overline{D}^{(i+1)}$ for every $0 \leq i \leq k-1$.
\begin{lemma}
Let $(X,d)$ be a metric space and $A \subset X $ be of the form
\begin{equation}
 A=\bigcups{i\geq1}B_{r_{i}}(x_{i}),
\end{equation}
where the $B_{r_{i}}(x_{i})$ are open disjoint balls with centers $x_{i}$ and radii
$r_{i}$, and let $\lims{i \gtinf} r_{i}=0$. Then we have
\begin{equation}
 \overline{A}=\bigcups{i \geq 1}\overline{B_{r_{i}}(x_{i})} \cup
H,
\end{equation}
where $H$ is the closure of the set $\sequence{x_{i}}{i\geq 1}$.
\end{lemma}
\begin{proof}
First we have $\bigcups{i \geq 1}\overline{B_{r_{i}}(x_{i})} \cup
H \subset  \overline{A}$, because the closure of a subset stays in the closure of the set
itself. Obviously, $A \subset \bigcups{i \geq 1}\overline{B_{r_{i}}(x_{i})}
\cup H $. Now for $y \in \overline{A}\backslash A$ we know that there exists a sequence
$\sequence{y_{n}}{n \geq 1} \subset A$ such that $\lims{n\gtinf}
d(y_{n},y)=0$. By the disjointness assumption one can distinguish two cases: either there
exists a $k \geq 1$ and an $i' \geq1$, such that $y_{n} \in B_{r_{i'}}(x_{i'})$ for all
$n\geq k$, and it follows that $y\in \overline{B_{r_{i'}}(x_{i'})}$; or for all $i \geq1$
$B_{r_{i}}(x_{i})$ contains at most a finite number of elements of the sequence
$\{y_{n}\}_{n\geq1}$. In the second case the sequence visits an infinite number of
disjoint balls. It follows by a $2\epsilon$-argument and the fact that the radii $r_{i}$
tend to zero that $y \in H$.
\end{proof}
We can apply the preceding lemma to each of the terms $\bigcups{1 \leq j
\leq k_i}(d^{i}_{j}-r^{m,i}_{j},d^{i}_{j}+r^{m,i}_{j})$ of the above union.
We know that  $\overline{D} \subset U_{m}$ and therefore the points of accumulation of
$\{d^{i}_{j}\}_{1 \leq j
\leq k_i,1\leq i \leq k-1}$ are already in
$\sbigcup{i=0}{k-1}\bigcups{1 \leq j
\leq k_i}[d^{i}_{j}-r^{m,i}_{j},d^{i}_{j}+r^{m,i}_{j}]$. Hence
\begin{equation}
 \overline{U_{m}}=\sbigcup{i=0}{k-1}\bigcups{1 \leq j
\leq k_i}[d^{i}_{j}-r^{m,i}_{j},d^{i}_{j}+r^{m,i}_{j}].
\end{equation}
We observe that $\lambda(U_1)< \infty$ and
$\lambda(U_{m})\leq\frac{1}{m}\lambda(U_1)$ holds for all $m\geq1$.
Additionally we have
$r^{m+1,i}_{j}<r^{m,i}_{j}$ for all  $m\geq 1$ and $0 \leq i \leq k-1,1  \leq j
\leq k_i$.
So one recognizes that $\{U_{m}\}_{m\geq1}$ satisfies the condition
\eqref{eq:approxsequence0}.\\
\subsubsection*{Construction of the sequence $\sequence{\alpha^{n}}{n\geq1}$}
First we fix an $m \geq1$ and consider the restriction $\alpha|_{[-m,m] \cap
U_{m}^{c}}$ of $\alpha$. By the Tietze extension theorem (p. 62, \cite{WalterAnalysis})
we can extend this restriction to a continuous function $\alpha^{(m)}$ defined on the whole of $\mathbb{R}$, while preserving its supremum and infimum.
One can also assume that $\alpha^{(m)}$ is constant on $(-\infty,-m)$ and
$(m,\infty)$.\\
Now we choose a sequence of smoothing kernels $\sequence{\phi_{k}}{k \geq1}$ as described
in the Appendix \ref{sec:convolution} and approximate $\alpha^{(m)}$ for a $k\geq1$ by the
convolution  $\alpha^{(m),k}:=(\alpha^{(m)}*\phi_{k})$.\\
The sequence $\sequence{\alpha^{(m),k}}{k\geq1}$ converges uniformly to $\alpha^{(m)}$
on the whole set $\Reals$ because of theorem \ref{thm:faltungapprox} and
\eqref{eq:localconvolution}.
Next for $\epsilon>0$ satisfying $\min
\{2-\sups{x \in \Reals}\alpha(x),\infs{x \in \Reals}\alpha(x)\} >
\epsilon$ and every $m \geq 1$ we can choose $k^{'}_{m} \geq 1$, such that
\begin{equation}
\begin{split}
\sups{x \in \Reals} \alpha^{(m),k}(x) \leq \sups{x \in \Reals}
  \alpha(x)+\epsilon<2\\
  \infs{x \in \Reals} \alpha^{(m),k}(x) \geq \infs{x \in \Reals}
\alpha(x)-\epsilon>0
\end{split}
\end{equation}
for all $k \geq k^{'}_{m}$.
Then we obtain
\begin{equation} \label{eq:S2proof}
 0 < \infs{m \geq 1} \infs{k \geq k^{'}_{m}}\infs{x \in \Reals}
\alpha^{(m),k}(x) \leq \sups{m \geq 1} \sups{k \geq k^{'}_{m}}\sups{x \in
\Reals} \alpha^{(m),k}(x)<2.
\end{equation}
All elements of $\sequence{\alpha^{(m),k}}{k\geq1}$ belong to
$\overline{C}^{1}(\Reals)$, because every element is constant on $(-\infty,-m-1/k)$ and
$(m+1/k,\infty)$ (cf. \eqref{eq:localconvolution}) and continuously
differentiable and bounded on every bounded interval. Now we want to construct the
sequence $\sequence{\alpha^{n}}{n\geq1}$ we are looking for.
For every $n \geq 1$ we choose $k_n \geq k^{'}_n$ such that
\begin{equation}\label{eq:S3proof}
 \sups{x \in [-n,n] \cap (U_{n})^{c}}
|\alpha^{(n),k_n}(x)-\alpha(x)|<1/n.
\end{equation}
We define
\begin{equation}
\alpha_{n}:=\alpha^{(n),k_{n}}, n\in \mathbb{N}.
\end{equation}
We notice that $\sequence{\alpha_{n}}{n\geq1}$ satisfies condition
\eqref{eq:approxsequence1} in consequence of \eqref{eq:S2proof} and
condition \eqref{eq:approxsequence2} because of \eqref{eq:S3proof}.
\subsubsection*{Application of the existence theorem}
We now deduce the existence of stable-like processes with discontinuous stability index function in the sense of a solution of the associated martingale problem. In fact, since the sequence of functions $\sequence{\alpha_{n}}{n\geq1}$ lies in
$\overline{C}^{1}(\Reals)$ and \eqref{eq:approxsequence1} holds, we can use theorem
\ref{thm:BassExistenceResult} and the remarks following the theorem to verify the
existence of Feller processes associated with the symbols
$q^{n}(x,\xi):=|\xi|^{\alpha^{n}(x)}, n\in
\mathbb{N},x,\xi \in \Reals.$
Let $\sequence{A_n}{n\geq1}$ denote the sequence of associated generators.
We want to show that the Feller processes satisfy the conditions of theorem
\ref{thm:centralresult}.
We define
\begin{equation}
 \underline{\alpha} := \infs{n \geq 1}\infs{x \in \Reals} \alpha^{n}(x)
~\text{and} ~\overline{\alpha} := \sups{n \geq 1}\sups{x \in \Reals}
\alpha^{n}(x).
\end{equation}
The symbols $q^{n}$ are all real-valued. For $|\xi| \geq 1, n\in
\mathbb{N}, x
\in
\mathbb{R},$ we have
\begin{equation}
 |\xi|^{\underline{\alpha}}\leq |\xi
|^{\alpha^{n}(x)} \leq  |\xi|^{\overline{\alpha}},
\end{equation}
and for $|\xi| \leq 1, n\in \mathbb{N}, x \in
\mathbb{R}$ the inequality
\begin{equation}
 |\xi|^{\overline{\alpha}}\leq |\xi
|^{\alpha^{n}(x)} \leq    |\xi|^{\underline{\alpha}}
\end{equation}
holds. So, by assumption \eqref{eq:approxsequence2} we see that the four conditions
\eqref{eq:tightnesscondition0},\eqref{eq:tightnessconditions1},
\eqref{eq:tightnessconditions2} and \eqref{eq:SymbolLowerBound} are satisfied.
Also $\mathcal{R}(A^{n}) \subset \overline{C}(\Reals), n\geq1$
(see section \ref{sec:stablelike}) and $C_{c}^{\infty}(\Reals) \subset
\Def(A^n),n\geq1$.\\
Now we want to show that for the sequence $\sequence{U_{m}}{m\geq1}$ of open sets we
defined in the preceding paragraphs the condition \eqref{eq:restricteduniformcont} holds.
Let $x \in [-m,m] \cap (U_{m})^{c}$ and $f \in \Def(A).$ Then for all
$x \in \Reals$ we have
\begin{equation}
\begin{split}
 &\, |(A^{n}-A )f(x)| \\
= &\, |\frac{1}{(2\pi)^{1/2}}\ints{\Reals}(e^{i x\xi}
(q^{n}(x,\xi)-q(x,\xi))\hat{f}(\xi)d\xi |\\
\leq &\, \frac{1}{(2\pi)^{1/2}}\ints{\Reals}
|q^{n}(x,\xi)-q(x,\xi)|{|\hat{f}(\xi)|}d\xi \\
= &\, \frac{1}{(2\pi)^{1/2}}\ints{\Reals}
\left||\xi|^{\alpha^{n}(x)}-|\xi|^{\alpha(x)}\right|{|\hat{f}(\xi)|}d\xi \\
= &\, \frac{1}{(2\pi)^{1/2}}\ints{\Reals}
\left|\sint{\alpha(x)}{\alpha^{n}(x)}|\xi|^{z}\log|\xi|
dz \right|{|\hat{f}(\xi)|} d\xi \\
\leq &\, \frac{1}{(2\pi)^{1/2}}\ints{\Reals}
|\alpha^{n}(x)-\alpha(x)|\left|(|\xi|^{\underline{\alpha}}+|\xi|^{\overline{
\alpha }
}) \log|\xi| \right|{|\hat{f}(\xi)|} d\xi \\
\leq &\, \frac{1}{(2\pi)^{1/2}}|\alpha^{n}(x)-\alpha(x)| \ints{\Reals}
\left|(|\xi|^{\underline{\alpha}}+|\xi|^{\bar{\alpha}
}) \log|\xi| \right|{|\hat{f}(\xi)|} d\xi \\
\leq &\,  |\alpha^{n}(x)-\alpha(x)| \ints{\Reals} C (1 + |\xi|^{2})
|\hat{f}(\xi)|
d\xi.
\end{split}
\end{equation}
The constant $C>0$ in the last line can be chosen independently of $x,n,m$.
For $m \geq 1$ and $n \gtinf$ the term $|\alpha^{n}(x)-\alpha(x)|$ converges uniformly
to zero on $[-m,m] \cap (U_{m})^{c}$. This implies that condition
\eqref{eq:restricteduniformcont} is also satisfied, and finishes the proof.
\end{proof}

\begin{remark}
The domain of definition of the operator can be extended by working with the bp-closure of
the linear span of $A$ (see proposition 3.4, chapter 4, \cite{EthierKurtz} and Appendix
3, \cite{EthierKurtz} for an introduction of the concept of bp-convergence).
\end{remark}

\begin{appendix}
\section{Topological concepts} \label{sec:topologicalconcepts}
Let $(X,d)$ be a metric space. We define the distance of a set $A \subset X$ to a point $x
\in X$ in the usual way $d(x,A) := \infs{y \in A} d(x,y)$.

Let $A \subset X$. We denote by $A^{(1)}:=\{ x \in X ~|~ d(x,A\backslash\{x\})=0\}$ the
set of all cluster points of $A$. We write $A^{(0)}:=\overline{A}$ for the
closure of $A$. We recursively define $A^{(k+1)}:=(A^{(k)})^{(1)}$
for all $k \geq 0$. And we call these sets the \emph{derived sets} of $A$.
For $k\in \mathbb{N} \cup 0$ one can see that $A^{(k)}\backslash
A^{(k+1)}$ has only isolated points and is closed.
The derived sets are hierarchically ordered:
\begin{equation}
 A^{(k+1)} \subset A^{(k)},\quad k\geq 0.
\end{equation}
A closed $A \subset X$ which satisfies $A \subset
A^{(1)}$ is said to be \emph{perfect}. One can show that $A^{(\infty)}:=\bigcaps{k \geq 0}
A^{(k)}$ is perfect. We will use the following theorem.
\begin{thm}\label{thm:CantorResultat}
A non-empty perfect subset of $\mathbb{R}^{d}$ is uncountable.
\end{thm}
\begin{proof}
See Theorem 2.43, \cite{RudinAnalysis}.
\end{proof}

\section{Convolutions and mollifiers}\label{sec:convolution}
Basic results on convolutions can be found for example in chapter VI.3, \cite{yosida}.
\begin{definition}
 We define the convolution f * g of a locally integrable function $f:\Rd \gt\Reals$
and a measurable function $g:\Rd \gt \Reals$  with compact support as follows:
\begin{equation}
 (f * g)(x):=\ints{\Rd}f(y)g(x-y)dy,x \in \Rd.
\end{equation}
\end{definition}
We define the support of a function $f:\Rd \gt \Reals$ by
\begin{equation}
 \operatorname{supp} f :=  \overline{\{ x \in \Rd ~|~ f(x) \neq 0 \}},
\end{equation}
e.g. the closure of the set of points, where $f$ is not zero.
\begin{definition}
By a \emph{mollifier} we understand a function $ \phi \in C^{\infty}(\Rd)$
which satisfies
\begin{itemize}
\item
\begin{equation}
 \operatorname{supp} \phi \subset \overline{B_{1}(0)},
\end{equation}
\item
\begin{equation}
 \int \phi d\lambda=1, \phi \geq 0.
\end{equation}
\end{itemize}
\end{definition}
One can derive further mollifiers from existing ones by defining
\begin{equation}
 \phi_{k}(x):= k\phi(kx), x \in \Rd,
\end{equation}
for all $k \geq 1$.
This is a sequence of mollifiers, whose support converges to $\{0\}.$ One can use such a sequence to approximate functions with the help of the
convolution. Here we consider only the one-dimensional case.
\begin{thm}\label{thm:faltungapprox}
For $K>0$, a continuous function $f:\Reals \gt \Reals$ can be uniformly approximated on
any interval $[-K,K]$ by the sequence $\sequence{(f * \phi_{k})}{k\geq1}$:
\begin{equation}
 \lims{k \gtinf} \sups{x \in [-K,K]}|(f * \phi_{k})(x)-f(x)|=0.
\end{equation}
\end{thm}
As another property of the convolution of a function $f:\Reals \gt \Reals$ with a mollifier
$\phi_{k},k \geq1,$ we note that that for $x \in \Reals$
\begin{equation}\label{eq:localconvolution}
 (f*\phi_{k})(x)= C, ~\text{if} ~ f|_{(x-1/k,x+1/k)}= C
~\text{for a}~ C
\in \Reals.
\end{equation}

\section{Poisson random measures}\label{sec:poissonpoint}
In this Appendix, we follow mainly section 2.3 in \cite{Applebaum}.
\begin{definition}
Let $(\Omega,\Filt,\mathbb{P})$ be a probability space.
A random measure on the measurable space $(S,\mathcal{S})$ is a function
\begin{equation}
M:\mathcal{S} \times \Omega \gt \Reals
\end{equation}
such that $M(\cdot,\omega)$ is a measure for every $\omega \in \Omega$ and
$M(B,\cdot)$ is a random variable for every $B \in \mathcal{S}$.
\end{definition}
In the following we write $M(B)$ for $M(B,\cdot)$.
Now we look at a special class of random measures -- Poisson random measures.
\begin{definition}
A Poisson random measure $N$ with intensity measure $\nu$ is a random measure on a
measurable space $(S,\mathcal{S})$ which satisfies the following conditions for a
non-trivial ring $\mathcal{A} \subset \mathcal{S}$:
\begin{itemize}
 \item if $A_{1},A_{2} \in  \mathcal{A}$ are disjoint, then $N(A_{1})$
and $N(A_{2})$ are independent,
\item for $A \in  \mathcal{A}$ the random variable $N(A)$ has a Poisson distribution with parameter $\nu(A)$.
\end{itemize}
\end{definition}
The following theorem will give an answer to the question, under which conditions an intensity measure
$\nu$ defines a Poisson random measure in a unique way.
\begin{thm}
Let $\nu$ be a $\sigma$-finite measure on a measurable space $(S,\mathcal{S})$. Then
there exists a unique in law Poisson random measure $M$ on a probability space $(\Omega,
\Filt,
\mathbb{P})$, such that
$\nu(A)=\E{M(A)}$ for all $A \in \mathcal{S}$. In this case
$\mathcal{A}=\{ A \in \mathcal{S} ~:~ \nu(A) < \infty\}$.
\end{thm}
\begin{proof}
See theorem 2.3.6, \cite{Applebaum}, and theorem 4.1 \cite{ItoPPP}.
\end{proof}
Two random measures $N_1,N_2$ on measurable
spaces $(S_1,\mathcal{S}_1),(S_2,\mathcal{S}_2)$ are said to be independent, if for all $A
\in \mathcal{S}_1,B \in \mathcal{S}_2$ the random variables $N_1(A)$
and $N_2(B)$ are independent.\\
The following assertions can be directly derived from the definition of Poisson random
measures.
If $N$ is a Poisson random measure on $S$ with intensity measure $\nu$ and
$B \in \mathcal{S}$, then $N^{B}(A):=N(B \cap A), A \in
\mathcal{S}$ defines another Poisson random measure with intensity
measure $\nu_{B}(A):=\nu(B \cap A), A \in \mathcal{S}$. Further let $B_1,B_2 \in
\mathcal{S}$ with $B_1 \cap B_2 = \emptyset$. Then the random measures $N^{B_1}$ and
$N^{B_2}$ are independent. Additionally, consider the case in which $S$ can be written as
the product of two measurable spaces $(S_1,\mathcal{S}_2)$ and
$(S_2,\mathcal{S}_2)$ such that $(S,\mathcal{S})=(S_1 \times
S_2,\mathcal{S}_1 \otimes \mathcal{S}_2)$ and $\nu= \nu_1 \otimes \nu_2$ with
 $\nu_2(S_2) < \infty$ and $\nu_1$ $\sigma$-finite. Then $N_1(A):=N(A \times
S_2), A \in \mathcal{S}_1,$ is a Poisson random measure on $S_1$ with intensity
measure $\nu^{2}(S_2)\nu^{1}$.\\
Now we define a compensated Poisson random measure.
\begin{definition}
 Let $\xi$ be a Poisson random measure and $\nu$ its intensity measure.
We call
\begin{equation}
 \tilde{\xi}:=\xi-\nu
\end{equation}
the associated compensated Poisson random measure.
\end{definition}
In the following $|\cdot|$ denotes the cardinality of a set and $\lambda$ once again
the Lebesgue measure.
A Poisson point process is defined as follows.
\begin{definition}
Let $(S,\mathcal{S})$ be a measurable space, $U:=\Reals^{+} \times
S$ and $\mathcal{U}:=\mathcal{B}([0,\infty)) \otimes
\mathcal{S}$. Let $(p_{t})_{t\geq0}$ be an adapted process on a filtered probability
space taking values in $S$ such that $\xi$ defined by $\xi([0,t] \times A):=|\{0\leq s <
t ~|~ \xi_{s} \in A\}|, t\geq0, A \in \mathcal{S},$ is a Poisson random measure
on $[0,\infty) \times S$ with intensity measure $\lambda \otimes \nu$. Then we say
that $(p_{t})_{t\geq0}$ is a Poisson point process with intensity measure $\nu$ and $\xi$
is its associated Poisson random measure.
\end{definition}
For the definition of stochastic integrals with respect to (compensated) Poisson random
measures we refer to \cite{Applebaum}.
\end{appendix}

\addcontentsline{toc}{section}{References}
\bibliographystyle{spmpsci}
\bibliography{PaperBib}


\end{document}